\documentclass[onefignum,onetabnum]{siamonline220329}

\usepackage{amsfonts,amssymb, amsmath, amsopn}
\usepackage{mathtools}
\usepackage{physics}
\usepackage[utf8]{inputenc}
\usepackage[autostyle]{csquotes}
\usepackage{tikz-cd}
\tikzstyle{arrow}=[draw, -latex] 
\usepackage{tabularx}
\usepackage{cleveref}
\usepackage{hyperref}
\usepackage{xcolor}
\usepackage{mwe}
\usepackage{zref-totpages}
\usepackage{float}
\usepackage{charter}
\usepackage{amsfonts}
\usepackage{graphicx}
\usepackage{epstopdf}
\usepackage{algorithmic}
\ifpdf
  \DeclareGraphicsExtensions{.eps,.pdf,.png,.jpg}
\else
  \DeclareGraphicsExtensions{.eps}
\fi

\usepackage{enumitem}
\setlist[enumerate]{leftmargin=.5in}
\setlist[itemize]{leftmargin=.5in}

\usepackage{alphabeta}
\makeatletter
\renewcommand\verbatim@font{\normalfont\fontencoding{T1}\ttfamily}
\makeatother
\usepackage{microtype} 
\usepackage[T1]{fontenc}
\usepackage{lmodern}

\usepackage{sansmathfonts}
\DeclareFontFamily{U}{futm}{}
\DeclareFontShape{U}{futm}{m}{n}{
  <-> s * [.92] fourier-bb
  }{}
\DeclareSymbolFont{Ufutm}{U}{futm}{m}{n}
\DeclareSymbolFontAlphabet{\mathbb}{Ufutm}

\crefname{section}{\S}{section}
\crefname{subsection}{\S}{subsection}
\newsiamthm{example}{\textbf{Example}}
\crefname{example}{Example}{Examples}
\newtheorem{question}{Question}[section]
\crefname{question}{Question}{Question}
\newsiamthm{notation}{\textbf{Notation}}
\crefname{notation}{Notation}{Notation}
\newsiamthm{remark}{\textbf{Remark}}
\crefname{remark}{Remark}{Remark}
\newsiamthm{fact}{\textbf{Fact}}
\crefname{fact}{Fact}{Fact}
\newsiamthm{property}{Property}
\crefname{property}{Property}{Property}
\newsiamthm{construction}{Construction}
\crefname{construction}{Construction}{Construction}
\newsiamthm{mytheorem}{\textbf{Theorem}}
\crefname{mytheorem}{\textsc{Theorem}}{Theorem}
\newsiamthm{mylemma}{\textbf{Lemma}}
\crefname{mylemma}{\textsc{Lemma}}{Lemma}
\newsiamthm{mydefinition}{\textbf{Definition}}
\crefname{mydefinition}{\textsc{Definition}}{Definition}
\newsiamthm{myproposition}{\textbf{Proposition}}
\crefname{myproposition}{\textsc{Proposition}}{Proposition}
\crefname{figure}{Figure}{Figure}
\crefname{example}{Example}{Example}

\headers{Construction of birational trilinear volumes via tensor rank criteria}{L. Bus\'e and P. Maz\'on}

\newcommand{\PPP}{(\P^1)^3}
\newcommand{\PPPe}{\P^1 \times \P^1 \times \P^1}

\newcommand{\R}{\mathbb{R}}
\newcommand{\C}{\mathbb{C}}

\newcommand{\Z}{\mathbb{Z}}
\newcommand{\A}{\mathbb{A}}

\renewcommand{\P}{\mathbb{P}}

\renewcommand{\rank}{\mathrm{rank}}

\newcommand{\mbf}{\mathbf}
\newcommand{\ovl}{\overline}
\newcommand{\bs}{\boldsymbol}
\newcommand{\dist}{\text{dist}}
\newcommand{\distbir}{\dist_\text{bir}}

\title{Construction of birational trilinear volumes via tensor rank criteria\thanks{Accepted for publication in the SIAM Journal on Applied Algebra and Geometry (SIAGA). 
Part of the results of the manuscript appear in the second author's Ph$.$D$.$ thesis \cite{thesis_poles}. 
\funding{The two authors were funded by the European Union’s Hori\-zon 2020 research and innovation programme, under the Marie Skłodowska-Curie grant agreement 860843. The second author was also funded by the European Union's Next Generation PRIN 2022, Prot$.$ 20223B5S8L.}}}

\author{
  Laurent Bus\'e\thanks{Université Côte d'Azur, Inria, 2004 route des Lucioles, 06902 Sophia Antipolis, France. (\email{laurent.buse@inria.fr}, \email{pablo.gonzalez-mazon@inria.fr}).}
  \and
  Pablo Maz\'on$^{\dagger}$\thanks{Dipartimento di Matematica, Università degli Studi di Trento, via Sommarive 14, I-38123 Povo di Trento (TN), Italy. (\email{pablo.gonzalezmazon@unitn.it}).}
}

\usepackage{amsopn}

\ifpdf
\hypersetup{
  pdftitle={Construction of birational trilinear volumes via tensor rank criteria},
  pdfauthor={Laurent Bus\'e and Pablo Maz\'on}
}
\fi

\makeatletter
\def\cref@getref#1#2{%
  \expandafter\let\expandafter#2\csname r@#1@cref\endcsname%
  \expandafter\expandafter\expandafter\def%
    \expandafter\expandafter\expandafter#2%
    \expandafter\expandafter\expandafter{%
      \expandafter\@firstoffive#2}}
\def\cpageref@getref#1#2{%
  \expandafter\let\expandafter#2\csname r@#1@cref\endcsname%
  \expandafter\expandafter\expandafter\def%
    \expandafter\expandafter\expandafter#2%
    \expandafter\expandafter\expandafter{%
      \expandafter\@secondoffive#2}}
      
\AtBeginDocument{%
   \def\label@noarg#1{%
    \cref@old@label{#1}%
    \@bsphack%
    \edef\@tempa{{page}{\the\c@page}}%
    \setcounter{page}{1}%
    \edef\@tempb{\thepage}%
    \expandafter\setcounter\@tempa%
    \cref@constructprefix{page}{\cref@result}%
    \protected@write\@auxout{}%
      {\string\newlabel{#1@cref}{{\cref@currentlabel}%
      {[\@tempb][\arabic{page}][\cref@result]\thepage}{}{}{}}}
    \@esphack}%
  \def\label@optarg[#1]#2{%
    \cref@old@label{#2}%
    \@bsphack%
    \edef\@tempa{{page}{\the\c@page}}%
    \setcounter{page}{1}%
    \edef\@tempb{\thepage}%
    \expandafter\setcounter\@tempa%
    \cref@constructprefix{page}{\cref@result}%
    \protected@edef\cref@currentlabel{%
      \expandafter\cref@override@label@type%
        \cref@currentlabel\@nil{#1}}%
    \protected@write\@auxout{}%
      {\string\newlabel{#2@cref}{{\cref@currentlabel}%
      {[\@tempb][\arabic{page}][\cref@result]\thepage}{}{}{}}}
    \@esphack}%
    } 

\begin{document}

\maketitle

\begin{abstract}
We provide effective methods to construct and manipulate trilinear birational maps $\phi:(\P^1)^3\dashrightarrow \P^3$ by establishing a novel connection between birationality and tensor rank. 
These yield four families of nonlinear birational transformations between 3D spaces that can be operated with enough flexibility for applications in computer-aided geometric design. 
More precisely, we describe the geometric constraints on the defining control points of the map that are necessary
for birationality, and present constructions for such configurations. 
For adequately constrained control points, we prove that birationality is achieved if and only if a certain $2\times 2\times 2$ tensor has rank one. 
As a corollary, we prove that the locus of weights that ensure birationality is $(\P_\R^1)^3$. 
Additionally, 
we provide formulas for the inverse $\phi^{-1}$ as well as the explicit defining equations of the irreducible components of the base loci.
Finally, we introduce a notion of distance to birationality for trilinear rational maps, and explain how to continuously deform birational maps.   
\end{abstract}



\section{Introduction}
\label{introduction}

In the fields of geometric modeling and computer-aided geometric design (CAGD), rational maps play a pivotal role. 
They offer an intuitive means of representing curves, surfaces, and volumes  \cite{cad_book_1,cad_book_2,
cox_applications_of_polynomials}, 
and have been instrumental for the 
modern development of these areas since the seminal works of Pierre Bézier and Paul de Casteljau in the 1950s and 1960s 
\cite{bezier_original,
de_Casteljau_original_1,
de_Casteljau_original_2}. 
The primary representation of rational parametrizations in practical scenarios involve control points, nonnegative weights, and blending functions. 
Moreover, the most frequent parametrizations rely on
tensor-product polynomials. 
We denote by $\P^n$ the complex projective $n$-space, and by $\P_\R^n$ the real projective $n$-space. 
For rational volumes, the typical parametrizations take the form
\begin{align}
\label{volume parametrization}
\phi : \P_\R^1 \times \P_\R^1 \times \P_\R^1 
&\dashrightarrow \P_\R^3 
\\[2pt] 
\nonumber
(s_0:s_1)\times (t_0:t_1)\times (u_0:u_1)
&\mapsto 
\sum_{i = 0}^n \sum_{j = 0}^m \sum_{k = 0}^l w_{ijk} \, \mbf{P}_{ijk} \, b_i^n(s_0,s_1) \, b_j^m(t_0,t_1) \, b_k^l(u_0,u_1) 
\ , 
\end{align}
for some \textit{control points} $\mbf{P}_{ijk} = (1,x_{ijk}, y_{ijk}, z_{ijk})$ in $\R^4$ and nonnegative \textit{weights} $w_{ijk}$ in $\R$, for each $0\leq i\leq n$, $0\leq j\leq m$, $0\leq k\leq l$, where 
$
b_i^n(s_0,s_1)$ is the $i$-th homogeneous Bernstein polynomial of degree $n$. 
The control points and weights offer intuitive insights into the geometry of the rational map. 
In a suitable affine chart, the parametrized shape mimics the net of control points, and the weights have a pull-push effect towards them (see e$.$g$.$ \cite{farin} and \cite[Chapter 3]{cox_applications_of_polynomials}). 
Furthermore, the control points provide useful differential information about the rational map.  

A rational parametrization is \textit{birational} if it admits an inverse map which is also rational
\cite{harris,HartshorneBook}. 
Birational maps have several advantages in applications. 
One key benefit is that they ensure global injectivity (on a Zariski open set). 
More importantly, the inverse can be exploited for computing preimages. 
Some applications require the computation of preimages for various purposes  \cite{computation_preimages_required_1,computation_preimages_required_2}, 
and it is convenient for others such as image and volume warping
\cite{
computation_preimages_warping_2,
computation_preimages_warping_3}, 
morphing \cite{computation_preimages_morphing_2}, 
texturing 
\cite{
computation_preimages_texture_2,
computation_preimages_texture_3}, 
or the generation of 3D curved meshes for geometric analysis 
\cite{computation_preimages_3D_curve_meshes_1,
computation_preimages_3D_curve_meshes_2}. 
Birational maps offer computational advantages since the
inverse yields formulas for these preimages 
without invoking numerical solving methods.

In general, \eqref{volume parametrization} is  not birational.  
More specifically, the locus of birational transformations in the space of such parametrizations generally has a large codimension (see e$.$g$.$ \cite{bisiplane,petitdegree,deserti_some_properties,cubic}). 
Therefore, birationality represents a notably restrictive condition, making the construction of birational maps a challenging algebraic problem that should be treated using tools from algebraic geometry and commutative algebra.

\subsection{Previous work}
\label{previous work}

Surprisingly, although birational geometry is a classical topic in algebraic geometry with a trajectory of over 150 years  (see e.g.~\cite{cremona1,hudson,noether-castelnuovo,Alberich,
petitdegree,dolgachev_cremona}), it wasn't until 2015 that the practical application of (nonlinear) birational transformations to design emerged, in the work of Thomas Sederberg and collaborators \cite{sederberg2D,SGW16}.  
The earliest works about  constructing birational maps dealt with 2D tensor-product parametrizations
$\phi:\P_\R^1\times \P_\R^1 \dashrightarrow \P_\R^2$ 
defined by polynomials of low degree. 
More general (nonrational) inversion formulas for bilinear rational maps are also studied in \cite{floater}.   
Recently, novel conditions for birationality have been proposed for parametrizations $\phi: \P_\R^2 \dashrightarrow \P_\R^2$ with quadratic entries, relying on the complex rational representation of a rational map \cite{construction_quadratic_Cremona_1,construction_quadratic_Cremona_2}. 

It is important to highlight three characteristics that are common to all the works addressing the construction of birational maps in CAGD published to date:
\begin{enumerate}
\label{principles for the construction of birational maps}
\item They only treat 2D parameterizations
\item They rely on the imposition of specific syzygies in order to achieve birationality  
\item They provide strategies for constructing (possibly constrained) nets of control points with sufficient flexibility, followed by the computation of weights that ensure birationality
\end{enumerate}

\subsection{Trilinear rational maps} 
\label{subsection: trilinear maps}

The multiprojective space $(\P_\R^1)^3$ (resp$.$ $(\P^1)^3$) is associated to the standard $\Z^3$-graded tensor-product ring $R = \R[s_0,s_1]\otimes \R[t_0,t_1]\otimes \R[u_0,u_1]$ (resp$.$ over $\C$).   
In this paper, we are interested in trilinear birational maps. 

\begin{definition}
A trilinear rational map is a rational map $\phi:(\P^1)^3\dashrightarrow\P^3$ defined by trilinear polynomials, i$.$e$.$ homogeneous polynomials in $R$ of $\Z^3$-degree $(1,1,1)$.  
\end{definition} 

A trilinear rational map can be defined by means of control points and weights as  
\begin{eqnarray}
\label{eq: trilinear map}
\phi : \P^1 \times \P^1 \times \P^1
&\dashrightarrow &\P^3 
\\[2pt] 
\nonumber
(s_0:s_1)
\times 
(t_0:t_1)
\times 
(u_0:u_1)
&\mapsto 
& 
(f_0:f_1:f_2:f_3)
=
\sum_{0 \leq i,j,k \leq 1} w_{ijk} \, \mbf{P}_{ijk} \, b_i^1(s_0,s_1) \, b_j^1(t_0,t_1) \, b_k^1(u_0,u_1) 
\end{eqnarray} 
where $b_0^1(s_0,s_1) = s_0 - s_1$ and $b_1^1(s_0,s_1) = s_1$. 

\subsection{Motivation and organization}

To motivate our work and contributions, we list four questions of interest for applications that are formalized and answered 
throughout the paper.

\begin{question}
\label{Q: constraints}
What constraints should be imposed on the control points to ensure the existence of weights that render $\phi$ birational?
\end{question}

\begin{question}
\label{Q: inverse}
If $\phi$ is birational, how can we compute $\phi^{-1}$?
\end{question}

\begin{question}
\label{Q: distance}
How far is $\phi$ from birationality? 
how can we compute a birational approximation?
\end{question} 

\begin{question}
\label{Q: deformation}
How can we deform birational maps while preserving birationality?
\end{question}

The first two questions are answered in \cref{sec: hexahedral} - \cref{sec: tripod}, namely in \cref{type pyramidal} - \cref{type tripod} and \cref{theorem: inverse hexahedral}, \cref{theorem: inverse pyramidal} - \cref{theorem: inverse tripod}, which represent the technical core of the paper. 
Each section follows the same structure, focusing on one of the four classes of trilinear birational maps: \textit{hexahedral}, \textit{pyramidal}, \textit{scaffold}, and \textit{tripod}, illustrated in \cref{fig: constraints}.
The main results of these sections, \cref{theorem: tensor hexahedral}, \cref{theorem: tensor pyramidal} - \cref{theorem: tensor tripod}, characterize birationality using a rank-one condition on a $2\times 2\times 2$ tensor. 
Remarkably, birationality can be determined using linear algebra methods, making it suitable for numerical computations. 
From our results, in \cref{sec: manip} we devise effective strategies for approximating and manipulating birational trilinear volumes, which allow us to answer the last two questions.
Additionally, in \cref{examples} we illustrate these methods through several examples, that emphasize how they are meaningful for applications.

\section{Preliminaries}
\label{preliminaries} 

In this section, we introduce the necessary notation and concepts for our analysis. 
We close it by proving a characterization of the existence of a linear syzygy between the defining polynomials of $\phi$, that relies on matrix rank. 

\subsection{Parametric and boundary varieties}
\label{definitions}

The concepts of parametric lines and surfaces help to describe the geometry of a trilinear rational map. 
We always work with the Zariski topology.

\begin{definition}
\label{parametric}
A parametric $s$-surface is the closure of $\phi\,(\P^1\times \P^1)$ for the specialization at $(s_0:s_1) = (\alpha_0:\alpha_1)$, for any parameter $(\alpha_0:\alpha_1)\in \P^1$.  
A parametric $s$-line is $\phi\,(\P^1)$ for the specializations $(t_0:t_1) = (\beta_0:\beta_1)$ and $(u_0:u_1) = (\gamma_0:\gamma_1)$, for any two parameters $(\beta_0:\beta_1)\in \P^1$ and $(\gamma_0:\gamma_1)\in \P^1$. 
With the obvious modifications, we define the $t$- and $u$-surfaces, as well as the $t$- and $u$-lines.
\end{definition}

It is straightforward to verify that the parametric surfaces are either planes or quadrics. More precisely, they correspond to the image of a bilinear parametrization $\P^1\times\P^1 \dashrightarrow \P^3$. 
 
Let $\A^3 \cong U\subset \PPPe$ be the affine space determined by $s_0\not=0$, $t_0\not=0$, and $u_0\not=0$. 
In this chart, the control points are the images of the vertices of the unit cube $[0,1]^3\subset U$. 
The concepts of boundary surfaces and lines are inherent to this point of view. 
To gain geometric intuition, \cref{fig: boundaries} illustrates these objects.

\begin{definition}
\label{def: boundaries}
The boundary surfaces, denoted by $\Sigma_i, T_j, Y_k$ for each $0\leq i,j,k\leq 1$, are the parametric surfaces defined by the supporting planes of the facets of the unit cube. 
The boundary lines, denoted by $s_{jk},t_{ik},u_{ij}$ for each $0\leq i,j,k\leq 1$, are the parametric lines defined by the supporting lines of the edges of the unit cube. 
\end{definition}

\begin{figure}
	\centering
	\includegraphics[width=\textwidth]{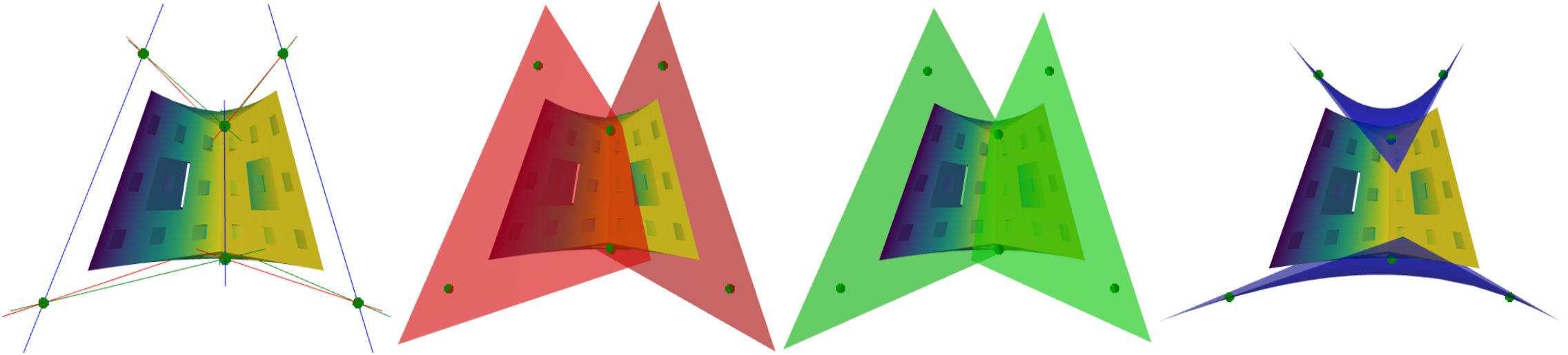}
	\caption{A deformation of a model, enclosed inside the unit cube, by means of a trilinear rational map. Leftmost: the boundary lines $s_{jk}$ (red), $t_{ik}$ (green), and $u_{ij}$ (blue), for each $0\leq i,j,k\leq 1$. From left to right: the boundary surfaces $\Sigma_0,\Sigma_1$ (red), $T_0,T_1$ (green), and $Y_0,Y_1$ (blue).}
	\label{fig: boundaries}
\end{figure}

Notice that, based on the previous observation about parametric surfaces, the boundary surfaces are always either planes or quadrics. 
It can occur that a boundary surface (resp$.$ line) is not a surface (resp$.$ line), but a curve (resp$.$ point) due to a contraction. 
To avoid this scenario, we always require the following nondegeneracy property. 
Notice that it is not at all strict, since it is satisfied by a general trilinear rational map.

\begin{property}
\label{property: nondegeneracy}
The following conditions are satisfied by a trilinear rational map:
\begin{enumerate}
\item For every $0\leq i,j,k\leq 1$, $\Sigma_i$, $T_j$, and $Y_k$ are smooth surfaces, pairwise distinct
\item For every $0\leq i,j,k\leq 1$, $s_{jk}$, $t_{ik}$, and $u_{ij}$ are lines, pairwise distinct
\end{enumerate}
\end{property}

In order to simplify the notation, we adopt the standard monomial basis in $R_{(1,1,1)}$ instead of the Bernstein basis. 
These two formulations are equivalent, and their analysis is hence identical. 

\begin{notation}
\label{notation}
We adopt the following conventions:
\begin{enumerate}
\setlength\itemsep{2pt}
\item We always assume that $\phi:\PPPe\dashrightarrow \P^3$ is the trilinear rational map defined by 
$$
\mbf{f} = (f_0,f_1,f_2,f_3) 
= 
\sum_{0\leq i,j,k \leq 1} w_{ijk} \, \mbf{P}_{ijk} \, s_i t_j u_k
$$
for some $\mbf{P}_{ijk} = (1,x_{ijk},y_{ijk},z_{ijk})$ in $\R^4$ and $w_{ijk}$ in $\R$, for each $0\leq i,j,k\leq 1$ 
\item $\Sigma_i$ (resp$.$ $T_j$ and $Y_k$) is the closure of $\phi(\P^1\times\P^1)$ for the specialization at $s_i = 0$ (resp$.$ $t_j = 0$ and $u_k = 0$), and $s_{jk}$ (resp$.$ $t_{ik}$ and $u_{ij}$) is $\phi(\P^1)$ for the specializations at $t_j = u_k = 0$ (resp$.$ $s_i = u_k = 0$ and $s_i = t_j = 0$)
\item If $\Sigma_i$ is a plane, it is defined by the linear form 
$
\sigma_i = \bs{\sigma}_i \cdot \mbf{x}^T 
$ 
for some $\bs{\sigma}_i = (\sigma_{0i}, \sigma_{1i}, \sigma_{2i}, \sigma_{3i})$ in $\R^4$. If it is a quadric, it is defined by the form $\sigma_i = \mbf{x} \cdot \bs{\sigma}_i \cdot \mbf{x}^T$ for some symmetric $4\times 4$ matrix in $\R^{4\times 4}$
\item Analogously, we denote by $\tau_j,\upsilon_k$ the defining polynomials of $T_j,Y_k$, and $\bs{\tau}_j,\bs{\upsilon}_k$ their associated vectors and symmetric matrices
\end{enumerate}
\end{notation}

A straightforward numerical criterion to determine whether \(\Sigma_i\) (or any other boundary surface) is either a plane or a quadric involves computing the rank of the \(4 \times 4\) matrix \((\mbf{P}_{i00}^T, \mbf{P}_{i10}^T, \mbf{P}_{i01}^T, \mbf{P}_{i11}^T)\):
if the rank is \(4\), \(\Sigma_i\) is a quadric, and if the rank is \(3\), \(\Sigma_i\) is a plane.  
Note that Property \(2.3\) ensures lower ranks are excluded. 

\subsection{Trilinear birational maps}
\label{trilinear birational}

Let $\mbf{x} = (x_0,x_1,x_2,x_3)$ be the homogeneous variables in $\P^3$. 
If $\phi$ is birational, the inverse $\phi^{-1}$ has the form 
\begin{eqnarray}
\label{eq: inverse map}
\phi^{-1} : \P^3 &\dashrightarrow & \P^1 \times \P^1 \times \P^1
\\[2pt] 
\nonumber
(x_0:x_1:x_2:x_3)
&\mapsto 
&
(A_0:
A_1)
\times
(B_0:
B_1)
\times
(C_0:
C_1)
\ , 
\end{eqnarray} 
where $A_i = A_i(\mbf{x})$ (resp$.$ $B_j = B_j(\mbf{x})$ and $C_k = C_k(\mbf{x})$) are homogeneous of the same degree for $i = 0,1$ (resp$.$ $j = 0,1$ and $k = 0,1$), without a common factor. 

\begin{definition}
If $\phi$ is birational, the type of $\phi$ is the triple $(\deg A_i, \deg B_j, \deg C_k)$ in $\Z^3$.
\end{definition} 

The defining polynomials of $\phi^{-1}$ are either linear or quadratic, and each type determines a class of trilinear birational maps  \cite{trilinear}. 
Hence, there are only four possibilities up to permutation: $(1,1,1)$, $(1,1,2)$, $(1,2,2)$, and $(2,2,2)$.  
On the other hand, the composition $\phi^{-1}\circ\phi :(\P^1)^3\dasharrow(\P^1)^3$ yields the identity on some open set of $\PPP$. 
Namely, $\phi^{-1}\circ \phi$ is generically given by
\begin{eqnarray*}
(s_0:s_1)
\times
(t_0:t_1)
\times
(u_0:u_1)
&
\mapsto
&
(A_0(\mbf{f}):A_1(\mbf{f}))
\times
(B_0(\mbf{f}):B_1(\mbf{f}))
\times
(C_0(\mbf{f}):C_1(\mbf{f}))
\ .
\end{eqnarray*} 
Therefore, the $2\times 2$ determinants in $R[x_0,x_1,x_2,x_3] = \R[s_0,s_1,t_0,t_1,u_0,u_1,x_0,x_1,x_2,x_3]$
\begin{equation}
\label{algebraic relations}
\begin{vmatrix}
s_0 & s_1 \\ 
A_0 & A_1
\end{vmatrix}
\ \ ,\ \
\begin{vmatrix}
t_0 & t_1 \\ 
B_0 & B_1
\end{vmatrix}
\ \ ,\ \
\begin{vmatrix}
u_0 & u_1 \\ 
C_0 & C_1
\end{vmatrix}
\end{equation}
vanish with the specializations $x_n \mapsto f_n$ for each $0\leq n\leq 3$. 
In particular, they represent algebraic relations satisfied by the defining polynomials of $\phi$.  
More explicitly, if $\sigma_0,\sigma_1$ (resp$.$ $\tau_0,\tau_1$ and $\upsilon_0,\upsilon_1$) are linear, then the first (resp$.$ second and third) relation is a syzygy of degree $(1,0,0)$ (resp$.$ $(0,1,0)$ and $(0,0,1)$). 
Recall that a tuple $\bs{\pi} = (\pi_0,\pi_1,\pi_2,\pi_3)$ in $R^4$ of homogeneous polynomials of the same degree is a \textit{syzygy} of $\mbf{f}$ if 
$$
\langle
\bs{\pi},\mbf{f}
\rangle 
= 
\bs{\pi} \cdot \mbf{f}^T
= 
\pi_0 f_0 + \pi_1 f_1 + \pi_2 f_2 + \pi_3 f_3 = 0
\ ,
$$ 
where $\langle -,-\rangle$ stands for the usual scalar product. 
The $\Z^3$-degree of the polynomials in $\bs{\pi}$ is the \textit{degree} of the syzygy. 
Equivalently, the syzygies of $\mbf{f}$ are identified with the linear polynomials in  $R[x_0,x_1,x_2,x_3]$ that vanish after the specializations $x_n \mapsto f_n$ for each $0\leq n\leq 3$. 
Hence, the first (resp$.$ second and third) determinant in  \cref{algebraic relations} can be regarded as a syzygy.
On the other hand, if the defining polynomials of $\phi^{-1}$ are quadratic, the determinants in \cref{algebraic relations} correspond to more complicated relations in the defining ideal of the Rees algebra associated to $\phi$ (see \cite{trilinear}).  

Geometrically, the determinants in \cref{algebraic relations} define the implicit equations of the parametric surfaces. 
More explicitly, given a general $(s_0:s_1) = (\lambda_0:\lambda_1)$ in $\P^1$, the equation of the corresponding $s$-surface in $\P^3$ is 
$$
\begin{vmatrix}
\lambda_0 & \lambda_1 \\ 
A_0 & A_1
\end{vmatrix}
= 
0
\ ,
$$
and similarly for the $t$- and $u$-surfaces. 
In particular, the  $s$-surfaces (resp$.$ $t$- and $u$-surfaces) form the  pencil, or linear system, of surfaces spanned by $\Sigma_0,\Sigma_1$ (resp$.$ $T_0,T_1$ and $Y_0,Y_1$).

\begin{figure}
	\centering
	\includegraphics[width=\textwidth]{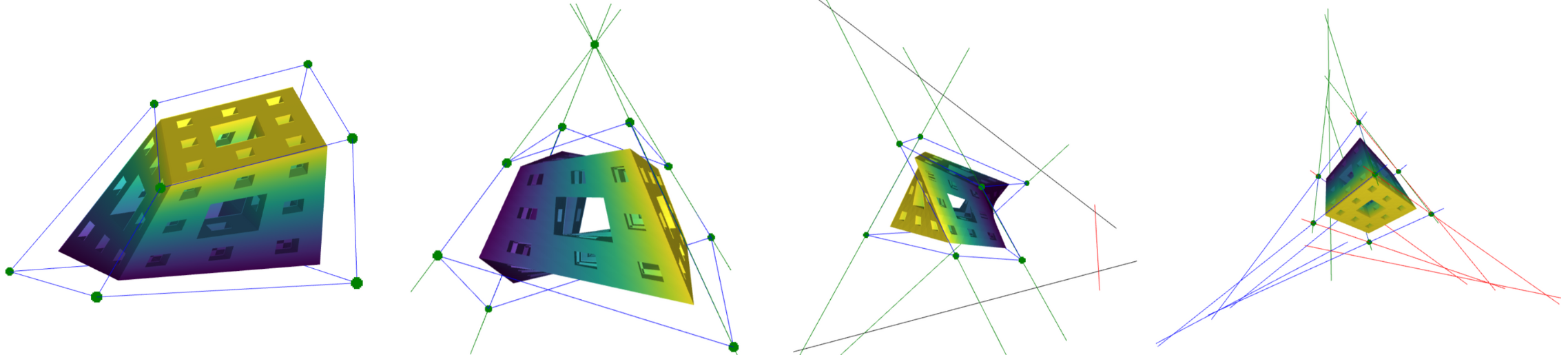}
	\caption{The distinct geometric constraints on the control points that are necessary for birationality. From left to right: hexahedral, pyramidal, scaffold, and tripod.}
	\label{fig: constraints}
\end{figure}

Regarding the construction of birational maps, the challenge lies in reconciling the necessary algebraic relations involved in birationality with the representation of $\phi$ relying on control points and weights. 
As we shall prove, these relations impose geometric constraints on the control points, that are formalized in the following four classes of trilinear rational maps: \textit{hexahedral}, \textit{pyramidal}, \textit{scaffold}, and \textit{tripod} (see \cref{def: hexahedral}, \cref{def: pyramidal}, \cref{scaffold}, and \cref{tripod}). 
Each of these classes are depicted in \cref{fig: constraints}, and respectively correspond to birational maps of type $(1,1,1)$, $(1,1,2)$, $(1,2,2)$, and $(2,2,2)$ (see \cref{type pyramidal}, \cref{type scaffold}, and \cref{type tripod}). 
More precisely, by \cite[Theorem 4.1]{trilinear} the locus of trilinear birational maps has eight irreducible components, and each component consists of the birational maps of a fixed type.  
In \cref{sec: hexahedral} - \cref{sec: tripod}, we prove that the loci of hexahedral, pyramidal, scaffold, and tripod birational maps contain open subsets of these irreducible components.

A natural question is whether maps from one class can be continuously deformed into another. The answer is yes. For example, starting from a pyramidal birational map, the control points of the quadric boundary surfaces can be specialized to planes, producing a map of type \((1,1,1)\). However, this degeneration is not a general hexahedral map, as the boundary lines for the previously quadratic parameter now intersect at a point, a condition inherited from pyramidal maps. In the paper, we focus on constructing general birational maps in each class, which are more relevant in applications. For a full classification of degenerations between classes, we refer the reader to \cite{trilinear}.

\subsection{Linear syzygies}
\label{linear syz} 
 
By ``linear syzygies'' we mean syzygies of degrees $(1,0,0)$, $(0,1,0)$, and $(0,0,1)$. 
The following lemma characterizes the existence of a linear syzygy of $\mbf{f}$ by means of a matrix rank condition. 
It is stated for syzygies of degree $(1,0,0)$, but it can be easily reformulated for syzygies of degrees $(0,1,0)$ and $(0,0,1)$ (see \cref{linear syz t u}). 

\begin{lemma} 
\label{lemma: linear syz s}
Let $\phi$ be dominant and $\Sigma_0,\Sigma_1$ be planes.
Then, $\mbf{f}$ has a syzygy of degree $(1,0,0)$ if and only if
\begin{equation}
\label{matrix: syz s}
\rank
\begin{pmatrix}
w_{100} \, \langle \boldsymbol{\sigma}_0 , \mbf{P}_{100} \rangle 
&
w_{110} \, \langle \boldsymbol{\sigma}_0 , \mbf{P}_{110} \rangle
&
w_{101} \, \langle \boldsymbol{\sigma}_0 , \mbf{P}_{101} \rangle
&
w_{111} \, \langle \boldsymbol{\sigma}_0 , \mbf{P}_{111} \rangle
\\
w_{000} \, \langle \boldsymbol{\sigma}_1 , \mbf{P}_{000} \rangle
&
w_{010} \, \langle \boldsymbol{\sigma}_1 , \mbf{P}_{010} \rangle
&
w_{001} \, \langle \boldsymbol{\sigma}_1 , \mbf{P}_{001} \rangle
&
w_{011} \, \langle \boldsymbol{\sigma}_1 , \mbf{P}_{011} \rangle
\end{pmatrix}
=
1
\ 
.
\end{equation}
Moreover, if \eqref{matrix: syz s} holds we find a point $(\alpha_0 : \alpha_1)$ in $\P_\R^1$ such that 
\begin{equation}
- \alpha_0 \, w_{1jk} \, \langle \boldsymbol{\sigma}_0 , \mbf{P}_{1jk} \rangle
=
\alpha_1 \, w_{0jk} \, \langle \boldsymbol{\sigma}_1 , \mbf{P}_{0jk} \rangle
\end{equation}
for every $0 \leq j,k \leq 1$. Then, any syzygy of degree $(1,0,0)$ of $\mbf{f}$ is proportional to 
\begin{equation}
\boldsymbol{\sigma} = \boldsymbol{\sigma}(s_0,s_1)
=
\alpha_0 \, \bs{\sigma}_0 \, s_0 
+ 
\alpha_1 \, \bs{\sigma}_1 \, s_1 
\ .
\end{equation}
\end{lemma}

\begin{proof}
For each $i = 0,1$, we define 
$
\mbf{g}_i 
= 
\mbf{g}_i (t_0,t_1,u_0,u_1)
=
\sum_{0\leq j,k \leq 1} w_{ijk} \, \mbf{P}_{ijk} \, t_j u_k
$,  
so that we can easily write $\mbf{f} = \mbf{g}_0 \, s_0 + \mbf{g}_1 \, s_1$. 
Suppose that $\mbf{f}$ has a syzygy of the form 
$
\boldsymbol{\pi} 
= 
\boldsymbol{\pi}(s_0,s_1)
=  
\bs{\pi}_0 
\, 
s_0
+  
\bs{\pi}_1 
\, 
s_1 
$
for some $\boldsymbol{\pi}_0,\boldsymbol{\pi}_1$ in $\R^4$.
Then, we find 
\begin{gather*}
\langle 
\boldsymbol{\pi} 
, 
\mbf{f}
\rangle 
= 
s_0
\, 
\langle 
\boldsymbol{\pi}_0 
, 
\mbf{f}
\rangle 
+ 
s_1 
\, 
\langle 
\boldsymbol{\pi}_1 
, 
\mbf{f}
\rangle 
= 
s_0^2
\, 
\langle 
\boldsymbol{\pi}_0 
, 
\mbf{g}_0
\rangle 
+ 
s_0 s_1 \, 
( 
\langle 
\boldsymbol{\pi}_0 
, 
\mbf{g}_1
\rangle 
+ 
\langle 
\boldsymbol{\pi}_1 
, 
\mbf{g}_0
\rangle 
)
+ 
s_1^2 
\, 
\langle 
\boldsymbol{\pi}_1 
, 
\mbf{g}_1
\rangle 
= 
0 
\ .
\end{gather*}
It follows immediately that $\langle \boldsymbol{\pi}_i , \mbf{g}_i \rangle = 0$ for both $i=0,1$. 
By hypothesis, the boundary surfaces $\Sigma_0,\Sigma_1$ are planes. 
Therefore, $\mbf{g}_i$ defines a bilinear rational parametrization $\P^1\times\P^1\dashrightarrow \Sigma_i$, and any vector $\bs{\pi}$ satisfying $\langle \bs{\pi} , \mbf{g}_i \rangle = 0$ must be proportional to $\bs{\sigma}_i$. 
In particular, we find $\boldsymbol{\pi}_i = \alpha_i \, \boldsymbol{\sigma}_i$ for some nonzero $\alpha_i$ in $\R$. 
Hence, we obtain 
\begin{gather*} 
\alpha_0 
\langle 
\boldsymbol{\sigma}_0 
, 
\mbf{g}_1
\rangle 
+ 
\alpha_1 
\langle 
\boldsymbol{\sigma}_1 
, 
\mbf{g}_0
\rangle 
=  
\sum_{0\leq j,k \leq 1} 
\left( 
\alpha_0 
\, 
w_{1jk}
\, 
\langle
\boldsymbol{\sigma}_0
,
\mbf{P}_{1jk}
\rangle
+
\alpha_1
\, 
w_{0jk}
\, 
\langle
\boldsymbol{\sigma}_1
,
\mbf{P}_{0jk}
\rangle
\right) 
t_j u_k
= 0
\ ,
\end{gather*} 
which is satisfied for some $\alpha_0,\alpha_1$ if and only if \cref{matrix: syz s} holds. Moreover, since $\phi$ is dominant no row in \cref{matrix: syz s} is identically zero.
Thus, $(\alpha_0:\alpha_1) \in \P_\R^1$ is unique, and $\boldsymbol{\pi}$ is proportional to $\boldsymbol{\sigma}$. 
\end{proof}

\begin{remark}
\label{linear syz t u}
After the obvious modifications, we derive analogous results to \cref{lemma: linear syz s} for syzygies of degrees $(0,1,0)$ and $(0,0,1)$. 
The rank conditions involved in the existence of the syzygies are respectively 
\begin{gather}
\label{matrix: syz t}
\rank
\begin{pmatrix}
w_{010} \, \langle \boldsymbol{\tau}_0 , \mathbf{P}_{010} \rangle
&
w_{110} \, \langle \boldsymbol{\tau}_0 , \mathbf{P}_{110} \rangle
&
w_{011} \, \langle \boldsymbol{\tau}_0 , \mathbf{P}_{011} \rangle
&
w_{111} \, \langle \boldsymbol{\tau}_0 , \mathbf{P}_{111} \rangle
\\
w_{000} \, \langle \boldsymbol{\tau}_1 , \mathbf{P}_{000} \rangle
&
w_{100} \, \langle \boldsymbol{\tau}_1 , \mathbf{P}_{100} \rangle
&
w_{001} \, \langle \boldsymbol{\tau}_1 , \mathbf{P}_{001} \rangle
&
w_{101} \, \langle \boldsymbol{\tau}_1 , \mathbf{P}_{101} \rangle
\end{pmatrix} 
= 
1
\ ,\\[4pt] 
\label{matrix: syz u}
\rank
\begin{pmatrix}
w_{001} \, \langle \boldsymbol{\upsilon}_0 , \mathbf{P}_{001} \rangle
&
w_{101} \, \langle \boldsymbol{\upsilon}_0 , \mathbf{P}_{101} \rangle
&
w_{011} \, \langle \boldsymbol{\upsilon}_0 , \mathbf{P}_{011} \rangle
&
w_{111} \, \langle \boldsymbol{\upsilon}_0 , \mathbf{P}_{111} \rangle
\\
w_{000} \, \langle \boldsymbol{\upsilon}_1 , \mathbf{P}_{000} \rangle
&
w_{100} \, \langle \boldsymbol{\upsilon}_1 , \mathbf{P}_{100} \rangle
&
w_{010} \, \langle \boldsymbol{\upsilon}_1 , \mathbf{P}_{010} \rangle
&
w_{110} \, \langle \boldsymbol{\upsilon}_1 , \mathbf{P}_{110} \rangle
\end{pmatrix}
=
1
\ .
\end{gather}
To elaborate, we demonstrate the equivalence between the existence of a syzygy of degree $(0,1,0)$ and \cref{matrix: syz t}.
Redefine $\mbf{g}_j = \mbf{g}_j(s_0,s_1,u_0,u_1) = \sum_{0\leq i,j\leq 1} w_{ijk} \, \mbf{P}_{ijk} \, s_i u_k$, and let $\bs{\pi} = \bs{\pi}_0 \, t_0 + \bs{\pi}_1 \, t_1$ be a syzygy of $\mbf{f}$. 
Following the lines of the proof of \cref{lemma: linear syz s}, the scalar product $\langle 
\bs{\pi} 
, 
\mbf{f}
\rangle$ equals 
\begin{equation}
\label{syzygy betas}
t_0 
t_1 
( 
\langle 
\boldsymbol{\pi}_0 
, 
\mbf{g}_1
\rangle 
+ 
\langle 
\boldsymbol{\pi}_1 
, 
\mbf{g}_0
\rangle 
)
= 
t_0 t_1
\sum_{0\leq i,k\leq 1} 
\left(
\beta_0
w_{i1k}
\langle 
\bs{\pi}_0 , 
\mbf{P}_{i1k}
\rangle 
+  
\beta_1
w_{i0k}
\langle 
\bs{\pi}_1 , 
\mbf{P}_{i0k}
\rangle 
\right)
s_i u_k
= 0
\ ,
\end{equation}
since the vanishing $\langle \bs{\pi} , \mbf{f}\rangle = 0$ readily implies $\bs{\pi}_j = \beta_j \, \bs{\tau}_j$ for some nonzero $\beta_0,\beta_1$. 
In particular, such $(\beta_0,\beta_1)\in \C^2$ exists if and only if \cref{syzygy betas} holds.
\end{remark}

\section{Hexahedral birational maps}
\label{sec: hexahedral}

In this section, we study the first and simplest family of trilinear birational maps: the class of hexahedral birational maps.  

\begin{definition}
\label{def: hexahedral}
A trilinear rational map is hexahedral if all the boundary surfaces are planes.
\end{definition}

Hexahedral birational maps are the direct generalization to 3D of the 2D quadrilateral birational maps studied in \cite{sederberg2D,EFFECTIVE_BIGRAD}, since the minimal graded free resolution of their base ideal is Hilbert-Burch \cite[Proposition 6.2]{trilinear}. 

\subsection{Geometric constraints}
\label{geometric constraint: hexahedral}

We begin describing the geometric constraints on the control points that are necessary for birationality. 
The following observation is straightforward.

\begin{fact}
\label{fact: hexahedral}
If $\phi$ is birational of type $(1,1,1)$, then it is also hexahedral.  
\end{fact} 

\begin{proof}
If $\phi$ is birational of type $(1,1,1)$ the polynomials $A_i$, $B_j$, and $C_k$ in \cref{algebraic relations} are all linear. 
Since the relations in \cref{algebraic relations} define the equations of the parametric surfaces for each of the three parameters, it follows that, for any specialization of these, the corresponding surfaces satisfy a linear equation. 
Therefore, all the parametric surfaces are planes. 
Since the boundary surfaces are also parametric, $\phi$ is hexahedral.
\end{proof}

In practice, the control points of a hexahedral map can be given indirectly through the boundary planes $\Sigma_i$, $T_j$, $Y_k$. 
More explicitly, we have the following construction. 

\begin{construction} 
\label{construction: hexahedral}
The control points of a hexahedral rational map can be generated as follows: 
\vskip2pt
\begin{enumerate}
\item As input, for each $0\leq i,j,k \leq 1$ choose pairwise distinct planes $\Sigma_i$, $T_j$, $Y_k$, or equivalently, vectors $\bs{\sigma}_i$, $\bs{\tau}_j$, $\bs{\upsilon}_k$ in $\R^4$, whose intersection is an affine point. 
\item As output, compute $\mbf{P}_{ijk} = \Sigma_i \cap T_j \cap Y_k$ for each $0\leq i,j,k\leq 1$
\end{enumerate}  
\end{construction} 

\subsection{Birational weights and tensor rank criterion}
\label{tensor criterion: hexahedral}

\begin{notation}
\label{hexahedral: deltas}
If $\phi$ is hexahedral, according to \cref{{construction: hexahedral}} each control point $\mbf{P}_{ijk}$ can be expressed as the wedge of the vectors $\bs{\sigma}_i$, $\bs{\tau}_j$, and $\bs{\upsilon}_k$ defining the boundary planes. More specifically, we have 
$
\mbf{P}_{ijk} = \Delta_{ijk}^{-1} \ \bs{\sigma}_i \wedge \bs{\tau}_j \wedge \bs{\upsilon}_k 
$, 
where 
\begin{equation}
\label{eq: Deltas hexahedral}
\Delta_{ijk} = 
\begin{vmatrix}
\sigma_{1i} & \sigma_{2i} & \sigma_{3i} \\ 
\tau_{1j} & \tau_{2j} & \tau_{3j} \\ 
\upsilon_{1k} & \upsilon_{2k} & \upsilon_{3k} 
\end{vmatrix} 
\end{equation}
for each $0\leq i,j,k \leq 1$. 
In particular, $\Delta_{ijk} = 0$ if and only if $\mbf{P}_{ijk}$ lies at the plane $x_0 = 0$ in $\P^3$, a possibility is excluded since by definition $\mbf{P}_{ijk} = (1,x_{ijk},y_{ijk},z_{ijk})$. 
\end{notation}

The following theorem establishes our birationality criterion for the class of hexahedral rational maps, based on a rank-one condition on a $2\times 2\times 2$ tensor. 
It relates birationality, the existence of linear syzygies, and tensor rank. 
Moreover, it is the key ingredient for our constructive results. 
Recall that a tensor $T$ in $\C^{r_1}\otimes\ldots\otimes\C^{r_n}$ has \textit{rank one} if $T = \bs{\alpha}_1\otimes \ldots \otimes \bs{\alpha}_n$ for some $\bs{\alpha}_i \in \C^{r_i}$. 
For a global perspective about tensor rank, we  refer the reader to  \cite{tensorbook,ottaviani_tensor_rank,CPD_2}. 

\begin{theorem}
\label{theorem: tensor hexahedral}
Let $\phi$ be hexahedral. 
Then, $\phi$ is birational of type $(1,1,1)$ if and only if the $2\times 2\times 2$ tensor 
\begin{equation}
W
=
\left( 
\frac{w_{ijk}}{\Delta_{ijk}}
\right)_{0\leq i,j,k \leq 1}
\end{equation} 
has rank one. 
\end{theorem} 

\begin{proof} 
By \cite[Theorem $6.1$]{trilinear}, it follows that $\phi$ is birational of type $(1,1,1)$ if and only if $\mathbf{f}$ admits syzygies of degrees $(1,0,0)$, $(0,1,0)$, and $(0,0,1)$. 
Equivalently, by \cref{lemma: linear syz s} and \cref{linear syz t u}, $\phi$ is birational if and only if the matrix rank conditions \cref{matrix: syz s}, \cref{matrix: syz t}, and \cref{matrix: syz u} are simultaneously satisfied. 
We now rewrite these rank conditions so that the matrices involved have the same entries. 
Namely, for any $0\leq i,j,k \leq 1$ we can write  
\begin{equation} 
\Delta_{0jk} 
\, 
\langle \boldsymbol{\sigma}_1 , \mathbf{P}_{0jk} \rangle
=
\boldsymbol{\sigma}_1
\wedge
\boldsymbol{\sigma}_{0}
\wedge
\boldsymbol{\tau}_{j}
\wedge
\boldsymbol{\upsilon}_{k}
= 
- 
\Delta_{1jk}
\,
\Delta_{1jk}^{-1}
\,
\boldsymbol{\sigma}_0
\wedge
\boldsymbol{\sigma}_1
\wedge
\boldsymbol{\tau}_j
\wedge
\boldsymbol{\upsilon}_k 
= 
- 
\Delta_{1jk} 
\, 
\langle 
\boldsymbol{\sigma}_0 
, 
\mathbf{P}_{1jk}
\rangle 
\ .
\end{equation}
Hence, we have 
$$
\begin{vmatrix}
\langle \bs{\sigma}_0 , \mbf{P}_{1jk} \rangle & - \langle \bs{\sigma}_1 , \mbf{P}_{0jk} \rangle\\
\Delta_{0jk} & \Delta_{1jk}
\end{vmatrix} 
= 0
\ ,
$$
and \cref{matrix: syz s} can be equivalently written as
\begin{equation}
\rank
\begin{pmatrix}
w_{100} \, {\Delta_{100}}^{-1}
&
w_{110} \, {\Delta_{110}}^{-1}
&
w_{101} \, {\Delta_{101}}^{-1}
&
w_{111} \, {\Delta_{111}}^{-1}
\\
w_{000} \, {\Delta_{000}}^{-1}
&
w_{010} \, {\Delta_{010}}^{-1}
&
w_{001} \, {\Delta_{001}}^{-1}
&
w_{011} \, {\Delta_{011}}^{-1}
\end{pmatrix}
=
1
\ .
\end{equation} 
With a similar argument, we derive 
$$
\Delta_{i0k} 
\, 
\langle 
\boldsymbol{\tau}_1 
, 
\mathbf{P}_{i0k}
\rangle 
= 
- 
\Delta_{i1k} 
\, 
\langle 
\boldsymbol{\tau}_0 
, 
\mathbf{P}_{i1k}
\rangle 
\ \ ,\ \ 
\Delta_{ij0} 
\, 
\langle 
\boldsymbol{\upsilon}_1 
, 
\mathbf{P}_{ij0}
\rangle 
= 
- 
\Delta_{ij1} 
\, 
\langle 
\boldsymbol{\upsilon}_0 
, 
\mathbf{P}_{ij1}
\rangle 
\ ,
$$
and \cref{matrix: syz t}, \cref{matrix: syz u} can be respectively rewritten as 
\begin{gather}
\rank
\begin{pmatrix}
w_{010} \, {\Delta_{010}}^{-1}
&
w_{110} \, {\Delta_{110}}^{-1} 
&
w_{011} \, {\Delta_{011}}^{-1}
&
w_{111} \, {\Delta_{111}}^{-1}
\\
w_{000} \, {\Delta_{000}}^{-1}
&
w_{100} \, {\Delta_{100}}^{-1}
&
w_{001} \, {\Delta_{001}}^{-1} 
&
w_{101} \, {\Delta_{101}}^{-1} 
\end{pmatrix}
=
1
\ ,\\
\rank
\begin{pmatrix}
w_{001} \, {\Delta_{001}}^{-1} 
&
w_{101} \, {\Delta_{101}}^{-1}
&
w_{011} \, {\Delta_{011}}^{-1}
&
w_{111} \, {\Delta_{111}}^{-1}
\\
w_{000} \, {\Delta_{000}}^{-1}
&
w_{100} \, {\Delta_{100}}^{-1}
&
w_{010} \, {\Delta_{010}}^{-1}
&
w_{110} \, {\Delta_{110}}^{-1}
\end{pmatrix}
=
1
\ .
\end{gather} 
In particular, the matrices involved in the three rank conditions above are the three flattenings of the tensor $W$ in the statement (see  \cite{tensorbook,ottaviani_tensor_rank,CPD_2}). 
By \cite[Proposition $3.11$]{ottaviani_tensor_rank}, they have all rank one if and only if $W$ has rank one, and the statement follows.
\end{proof} 

\begin{corollary} 
\label{corollary: weight formulas hexahedral}
Let $\phi$ be hexahedral. Then, $\phi$ is birational of type $(1,1,1)$ if and only if 
\begin{equation}
\label{eq: weight formulas hexahedral}
w_{ijk} = \alpha_i \, \beta_j \, \gamma_k \, \Delta_{ijk}
\end{equation}
for some $(\alpha_0:\alpha_1)\times (\beta_0:\beta_1)\times (\gamma_0:\gamma_1)$ in $(\P_\R^1)^3$. 
\end{corollary}

The choice of weights presented in \cref{corollary: weight formulas hexahedral}, as well as those for subsequent classes discussed in the following sections, can be understood through the following geometric interpretation.
The inverse rational map for each parameter corresponds to the linear parametrization of the parametric surfaces, explicitly given by \cref{algebraic relations}, as detailed in \cref{trilinear birational}. 
These parametrizations are determined by the images of three distinct points: \((1:0)\) and \((0:1)\),  which by assumption always map to the boundary surfaces, and a third point determined by the weights. 
Ensuring the rank-one condition guarantees that the three pencils of surfaces satisfy the algebraic relations necessary for birationality. 
Specifically, the weights determine the base locus of \(\phi^{-1}\), and the related contractions and blow-ups.

\subsection{Inverse and base locus}
\label{inverse: hexahedral}

We now address the most relevant computational advantage of birational maps: the efficient and exact computation of preimages. 
The following result presents explicit formulas for the inverse of a hexahedral birational map. 
In the remainder of the paper, we say that $\phi$ \textit{contracts}, or \textit{blows down}, a variety $V \subset (\P^1)^3$ if $\dim V > \dim \phi(V)$.

\begin{theorem} 
\label{theorem: inverse hexahedral}
Let $\phi$ be hexahedral with weights as \cref{eq: weight formulas hexahedral}. 
Then, $\phi^{-1}$ is given by 
\begin{equation}
\label{eq: inverse formula hexahedral}
(x_0:x_1:x_2:x_3) \mapsto 
(
-\alpha_1 \, \sigma_1 : \alpha_0 \, \sigma_0
)
\times 
(
-\beta_1 \, \tau_1 : \beta_0 \, \tau_0
)
\times 
(
-\gamma_1 \, \upsilon_1 : \gamma_0 \, \upsilon_0
)
\ .
\end{equation}
Additionally, we have the following:
\vskip2pt 
\begin{enumerate}
\item The base locus of $\phi$ is the curve $C$ in $(\P^1)^3$ defined by the ideal  
\begin{gather*}
\label{eq: base ideal hexahedral}
\left(
\sum
_{
0\leq j,k\leq 1}
w_{0jk}
\langle 
\bs{\sigma}_{1}
,
\mbf{P}_{0jk}
\rangle
\, t_{j} u_{k}
\, ,\, 
\sum
_{
0\leq i,k\leq 1}
w_{i0k}
\langle 
\bs{\tau}_{1}
,
\mbf{P}_{i0k}
\rangle
\, s_{i} u_{k}
\, ,\, 
\sum
_{
0\leq i,j\leq 1}
w_{ij0}
\langle 
\bs{\upsilon}_{1}
,
\mbf{P}_{ij0}
\rangle
\, s_{i} t_{j}
\right)
\end{gather*}
\item The base locus of $\phi^{-1}$ is the union of the lines $s = \Sigma_0 \cap \Sigma_1$, $t = T_0 \cap T_1$, and $u = Y_0\cap Y_1$
\item $\phi$ contracts the surfaces in $\PPP$ defined by the first, second, and third generators of the ideal in $1$ respectively to the lines $s,t$, and $u$
\item If $s,t$, and $u$ are pairwise skew, $\phi$ blows up $C$ to the unique (smooth) quadric through them
\end{enumerate}
\end{theorem}

\begin{proof}
By \cref{corollary: weight formulas hexahedral}, a hexahedral rational map is birational if and only if the weights satisfy \cref{eq: weight formulas hexahedral} for some $(\alpha_0:\alpha_1)\times (\beta_0:\beta_1)\times (\gamma_0:\gamma_1)$ in $(\P_\R^1)^3$.  
In particular, we can write 
\begin{gather*}
\mbf{f} 
= 
\sum_{0\leq i,j,k \leq 1} 
w_{ijk} 
\, 
\mbf{P}_{ijk} 
\, 
s_i t_j u_k
= 
\sum_{0\leq i,j,k \leq 1} 
\left( 
\alpha_i
\, 
\beta_j
\, 
\gamma_k
\, 
\Delta_{ijk}
\right)
\, 
\left( 
\Delta_{ijk}^{-1}
\, 
\bs{\sigma}_i
\wedge
\bs{\tau}_j
\wedge 
\bs{\upsilon}_k
\right)
\, 
s_i t_j u_k 
= 
\\
\sum_{0 \leq i,j,k \leq 1} 
(\alpha_i \, \bs{\sigma}_i \, s_i)
\wedge
(\beta_i \, \bs{\tau}_j \, t_j)
\wedge 
(\gamma_k \, \bs{\upsilon}_k \,u_k) 
= 
\bs{\sigma}(s_0,s_1) 
\wedge 
\bs{\tau}(t_0,t_1) 
\wedge 
\bs{\upsilon}(u_0,u_1) 
\ ,
\end{gather*} 
where $\bs{\sigma}=\bs{\sigma}(s_0,s_1)$ is as in \cref{lemma: linear syz s} and $\bs{\tau}=\bs{\tau}(t_0,t_1)$, $\bs{\upsilon}=\bs{\upsilon}(u_0,u_1)$ are analogous. 
Therefore, we find the relations 
$$
\langle
\bs{\sigma} 
, 
\mbf{f} 
\rangle 
= 
\bs{\sigma} 
\wedge 
\bs{\sigma} 
\wedge 
\bs{\tau} 
\wedge
\bs{\upsilon} 
= 
0 
\ ,\
\langle
\bs{\tau} 
, 
\mbf{f} 
\rangle 
= 
\bs{\tau}
\wedge 
\bs{\sigma} 
\wedge 
\bs{\tau} 
\wedge
\bs{\upsilon} 
= 
0 
\ ,\ 
\langle
\bs{\upsilon} 
, 
\mbf{f} 
\rangle 
= 
\bs{\upsilon}
\wedge 
\bs{\sigma} 
\wedge 
\bs{\tau} 
\wedge
\bs{\upsilon} 
= 
0 
\ ,
$$
yielding independent linear syzygies. 
Therefore, $\phi^{-1}$ is given by \cref{eq: inverse formula hexahedral}. 

Now, we prove the remaining four claims. 
Let $i^*$ stand for the converse of a binary index, namely $0^* = 1$ and $1^* = 0$. 
For each $i = 0,1$, we can write 
\begin{equation}
\label{eq: pullback of matrix syz}
\langle 
\boldsymbol{\sigma}_i 
, 
\mathbf{f} 
\rangle 
=
s_{i^*}
\sum_{0\leq j,k \leq 1}
w_{i^*jk} 
\, 
\langle 
\boldsymbol{\sigma}_i 
, 
\mathbf{P}_{i^*jk}
\rangle 
\, 
t_j
\, 
u_k 
= 
s_{i^*}
\, 
g_i(t_0,t_1,u_0,u_1)
\end{equation}
since by definition $\langle \boldsymbol{\sigma}_i , \mathbf{P}_{ijk}\rangle = 0$ for any $0\leq j,k\leq 1$.
The coefficients of $g_i$ (in the monomial basis) coincide with the entries of the $(i+1)$-th row of the matrix in \cref{matrix: syz s}. 
Since $\mbf{f}$ admits a syzygy of degree $(1,0,0)$, by \cref{lemma: linear syz s} the rank condition \cref{matrix: syz s} holds. 
Hence, $g_0,g_1$ are proportional to some bilinear $g$ and 
$$
(
\langle \bs{\sigma}_0,\mbf{f}\rangle 
,
\langle \bs{\sigma}_1,\mbf{f}\rangle 
)
= 
(\sigma_0(f_0,f_1,f_2,f_3)
,
\sigma_1(f_0,f_1,f_2,f_3))
: \mathfrak{N}^\infty
= 
(g)
\ ,
$$
where $\mathfrak{N}=(s_0,s_1)\cap(t_0,t_1)\cap(u_0,u_1)$ is the irrelevant ideal in $R$. 
Geometrically, this means that $\phi$ contracts the surface in $(\P^1)^3$ defined by $g = 0$ to the line $s = \Sigma_0 \cap \Sigma_1$, defined by the ideal $(\sigma_0,\sigma_1)\subset \R[x_0,x_1,x_2,x_3]$. 
With parallel arguments, it follows that $\phi$ contracts the surfaces defined by the second and third generators of the ideal in $1$ to the lines $t$ and $u$, respectively.
By \cref{eq: pullback of matrix syz} and the analogous identities for the other two pairs of variables, this ideal is precisely the saturation of $B = (f_0,f_1,f_2,f_3)$ and defines the base locus of $\phi$. 
Furthermore, the base ideal (resp$.$ locus) of $\phi^{-1}$ is  
$
(\sigma_0,\sigma_1) \cap (\tau_0,\tau_1) \cap (\upsilon_0,\upsilon_1)
\subset \R[x_0,x_1,x_2,x_3]
$
(resp$.$ $s\cup t\cup u$ in $\P^3$). 
Finally, since the composition $\phi\circ \phi^{-1}:\P^3\dashrightarrow\P^3$ yields the identity in an open set, we find 
$$ 
(
f_0(A_0,A_1,B_0,B_1,C_0,C_1)
\, ,\, 
\ldots 
\, ,\, 
f_3(A_0,A_1,B_0,B_1,C_0,C_1)
)
=
q\cdot 
(
x_0,x_1,x_2,x_3
)
$$ 
for some quadratic form $q = q(x_0,x_1,x_2,x_3)$ that vanishes at the three lines $s,t$, and $u$. 
If these are pairwise skew, then $q = 0$ is the unique quadric through the lines.
\end{proof}

\section{Pyramidal birational maps} 
\label{sec: pyramidal}

We now study the second family of birational maps: the class of pyramidal birational maps. 
This name is motivated by the geometry of the control points.

\begin{definition}
\label{def: pyramidal}
A trilinear rational map is pyramidal if for one of the three parameters:
\begin{enumerate}
\item The four boundary lines intersect at a point
\item The two boundary surfaces are quadrics
\end{enumerate} 
\end{definition}

\subsection{Geometric constraint}
\label{geometric constraint: pyramidal}

First, we prove that a general birational map of type either $(1,1,2)$, $(1,2,1)$, or $(2,1,1)$ is pyramidal.
Without loss of generality, we can restrict to $(1,1,2)$. 

\begin{proposition} 
\label{type pyramidal}
If $\phi$ is a general birational map of type $(1,1,2)$, then it is pyramidal. 
Moreover, 
$$
Y_0 \cap Y_1 = s \cup t \cup C
\ ,
$$ 
where $s = \Sigma_0 \cap \Sigma_1$, $t = T_0\cap T_1$, and $C$ is a plane conic that intersects both $s$ and $t$. 
\end{proposition} 

\begin{proof}
Since $\phi$ is general, \cref{property: nondegeneracy} is satisfied. 
Moreover, by \cite[Theorem $5.4$]{trilinear} we find 
\begin{equation} 
(f_0 , f_1 , f_2 , f_3)  : \mathfrak{N}^\infty = (a_0 , b_0 , c_0) \cap (h , c_1)
\end{equation}
for some linear $a_0 = a_0(s_0,s_1)$, $b_0 = b_0(t_0,t_1)$, $c_k = c_k(u_0,u_1)$ for $k = 0,1$, and bilinear irreducible $h = h(s_0,s_1,t_0,t_1)$, where $\mathfrak{N}\subset R$ is the irrelevant ideal.  
Let $\phi^{-1}$ be defined by 
$$
(x_0,x_1,x_2,x_3) \mapsto (A_0:A_1)\times (B_0:B_1)\times (C_0:C_1)
$$
for some linear $A_i = A_i(\mbf{x})$, $B_j = B_j(\mbf{x})$ and quadratic $C_k = C_k(\mbf{x})$. 
Since $\phi\circ\phi^{-1}:\P^3 \dashrightarrow \P^3$ yields the identity in an open set, we find 
$$
(f_0(A_0,A_1,B_0,B_1,C_0,C_1) \, , \, \ldots \, , \, f_3(A_0,A_1,B_0,B_1,C_0,C_1)) 
= 
F\cdot 
(x_0,x_1,x_2,x_3)
$$ 
for some cubic form $F = F(\mbf{x})$. 
Therefore, the pullback ideals 
$$
(a_0(A_0,A_1),b_0(B_0,B_1),c_0(C_0,C_1)) 
\ ,\ 
(h(A_0,A_1,B_0,B_1),c_1(C_0,C_1))
$$ 
are principal in $\R[x_0,x_1,x_2,x_3]$, and respectively generated by a linear and a quadratic form. 
In particular, the linear polynomials $a_0(A_0,A_1)$ and $b_0(B_0,B_1)$ are proportional, and define a plane $\Pi_0$. 
Geometrically, this means that the boundary planes $\Sigma_0,\Sigma_1, T_0$, and $T_1$ are not independent, and they intersect at a point $\mbf{A}$ in $\P_\R^3$. Hence, the four boundary $u$-lines, i$.$e$.$ $u_{ij} = \Sigma_i \cap T_j$ for each $0\leq i,j\leq 1$, intersect at $\mbf{A}$ and $\phi$ is pyramidal. 

Additionally, the quadratic polynomial $c_0(C_0,C_1)$ is divisible by $a_0(A_0,A_1)$ and $b_0(B_0,B_1)$. 
Therefore, $c_0(C_0,C_1)$ determines a rank-two quadric $\Pi_0\cup\Pi_1$ in the pencil of $u$-surfaces. 
On the other hand, $c_1(C_0,C_1)$ and $q = h(A_0,A_1,B_0,B_1)$ are also  proportional, and define a quadric $Q$. 
In particular, $Q$ contains the lines $s,t\subset \Pi_0$ since $h(A_0,A_1,B_0,B_1)$ vanishes at both. 
Hence, they lie in $Y_0\cap Y_1$. 
Furthermore, we have 
$$
Y_0\cap Y_1 = (\Pi_0\cup\Pi_1) \cap Q = (\Pi_0 \cap Q) \cup (\Pi_1\cap Q) 
= 
s \cup t \cup C
\ ,
$$ 
where $\Pi_1\cap Q = C$ is a plane conic that intersects both $s$ and $t$. 
Moreover, we can write   
$$
\nabla q
= 
\begin{pmatrix}
\frac{\partial h}{\partial s_0} 
& 
\frac{\partial h}{\partial s_1} 
&  
\frac{\partial h}{\partial t_0} 
& 
\frac{\partial h}{\partial t_1} 
\end{pmatrix} 
(A_0,A_1,B_0,B_1) 
\cdot 
D
$$
where $D$ is the Jacobian matrix of the tuple $(A_0,A_1,B_0,B_1)$ in $\R[x_0,x_1,x_2,x_3]^4$, which is scalar-valued. 
Since the partial derivatives of $h$ are homogeneous, $\nabla q$ vanishes at the point $\mbf{A}$ defined by the ideal $(A_0,A_1,B_0,B_1)$. 
Hence, $q$ has a singular point at $\mbf{A}$ and the quadric $Q$ is a cone. 
\end{proof}

In applications, the control points of a pyramidal map can be generated as follows. 

\begin{construction}
\label{construction: pyramidal}
The control points of a pyramidal rational map can be generated as follows: 
\begin{enumerate}
\item Choose a point $\mbf{A}$ in $\P_\R^3$ 
\item For each $0\leq i,j \leq 1$, choose pairwise distinct affine lines $u_{ij}$ in $\P_\R^3$  through $\mbf{A}$ 
\item Choose two distinct affine points $\mbf{P}_{ij0}$ and $\mbf{P}_{ij1}$ on $u_{ij}$, different from $\mbf{A}$
\end{enumerate}
\end{construction}

\subsection{Birational weights and tensor rank criterion}
\label{tensor criterion: pyramidal}

In all the upcoming statements of \cref{sec: pyramidal}, when $\phi$ is pyramidal we assume that $Y_0,Y_1$ are the quadric boundary surfaces. 

\begin{notation}
\label{notation: pyramidal}
Let $\phi$ be pyramidal. 
We denote by 
$\Pi_0$ the plane through the lines $s = \Sigma_0 \cap \Sigma_1$ and $t = T_0 \cap T_1$, which is defined by $\pi_0 = \bs{\pi}_0 \cdot \mbf{x}^T$ for some $\bs{\pi}_0$ in $\R^4$. 
In particular, we find points $(\lambda_0:\lambda_1)$ and $(\mu_0:\mu_1)$ in $\P_\R^1$ such that 
$
\pi_0 
= 
\lambda_0 \, \sigma_0 
+ 
\lambda_1 \, \sigma_1 
= 
\mu_0 \, \tau_0 
+ 
\mu_1 \, \tau_1  
$.
Additionally, for each $0\leq i,j,k \leq 1$ we set 
$$ 
\Delta_{ijk} = 
\frac{1}{
\langle \bs{\pi}_0 , \mbf{P}_{ijk} \rangle
}
\ .
$$ 
\end{notation}

\begin{remark}
For pyramidal maps, $\Delta_{ijk}$ is well-defined for every $0\leq i,j,k\leq 1$. 
Specifically, 
$\Delta_{ijk}$ is not well-defined if and only if $\mbf{P}_{ijk} = \mbf{A}$, a possibility that is excluded by \cref{property: nondegeneracy}. 
More explicitly, if $\mbf{P}_{ijk} = \mbf{A}$ then we find 
$
s_{jk} 
= 
\ovl{\mbf{P}_{0jk}\mbf{P}_{1jk}} 
= 
\ovl{\mbf{A}\mbf{P}_{i^*jk}}
= 
u_{i^*j}
$, 
as well as 
$
t_{ik} 
= 
u_{ij^*}
$.  
\end{remark}

The following is the analogous birationality criterion to \cref{theorem: tensor hexahedral} for pyramidal maps.

\begin{theorem} 
\label{theorem: tensor pyramidal}
Let $\phi$ be pyramidal. Then, $\phi$ is birational of type $(1,1,2)$ if and only if the tensor 
$$
W = \left(
\frac{w_{ijk}}{\Delta_{ijk}}
\right)_{0\leq i,j,k\leq 1}
$$ 
has rank one. 
\end{theorem}

\begin{proof}
By \cite[Theorem $6.1$]{trilinear}, $\phi$ is birational of type $(1,1,2)$ if and only if $\mathbf{f}$ admits syzygies of degrees $(1,0,0)$, $(0,1,0)$, but not $(0,0,1)$.
Since $Y_0$ and $Y_1$ are smooth quadrics, $\mbf{f}$ cannot have a syzygy of degree $(0,0,1)$. 
Therefore, by \cref{lemma: linear syz s} and \cref{linear syz t u} $\phi$ is birational if and only if the matrix rank conditions \cref{matrix: syz s} and \cref{matrix: syz t} hold simultaneously. 
Once more, we rewrite these rank conditions in terms of the flattenings of $W$. 

Recall that we let $0^* = 1$ and $1^* = 0$.  
For each $0\leq i,j,k \leq 1$, we can write   
\begin{gather*} 
\lambda_{i^*}
\langle 
\bs{\sigma}_{i^*} 
, 
\mbf{P}_{ijk}
\rangle 
= 
\langle 
\lambda_{0} 
\,
\bs{\sigma}_{0} 
+
\lambda_{1} 
\,
\bs{\sigma}_{1}  
, 
\mbf{P}_{ijk}
\rangle 
= 
\langle 
\bs{\pi}_0 
, 
\mbf{P}_{ijk}
\rangle 
= 
\Delta_{ijk}^{-1}
\ ,\\ 
\mu_{j^*}
\langle 
\bs{\tau}_{j^*} 
, 
\mbf{P}_{ijk}
\rangle 
= 
\langle 
\mu_{0} 
\,
\bs{\tau}_{0} 
+
\mu_{1} 
\,
\bs{\tau}_{1}  
, 
\mbf{P}_{ijk}
\rangle 
=
\langle 
\bs{\pi}_0 
, 
\mbf{P}_{ijk}
\rangle 
= 
\Delta_{ijk}^{-1}
\ .
\end{gather*} 
Therefore, \cref{matrix: syz s} and \cref{matrix: syz t} are respectively equivalent to 
$$
\rank
\begin{pmatrix}
w_{1jk} \, \Delta_{1jk}^{-1} \\ 
w_{0jk} \, \Delta_{0jk}^{-1}
\end{pmatrix}_{0\leq j,k \leq 1}
= 1
\ \ ,\ \  
\rank
\begin{pmatrix}
w_{i1k} \, \Delta_{i1k}^{-1} \\ 
w_{i0k} \, \Delta_{i0k}^{-1}
\end{pmatrix}_{0\leq i,k \leq 1}
= 1
\ ,
$$ 
which are rank-one conditions on the first two flattenings of $W$. 
By \cite[Proposition $3.11$]{ottaviani_tensor_rank},  $W$ has rank one if and only if these two conditions are satisfied, and the result follows. 
\end{proof} 

\begin{corollary} 
\label{corollary: weight formulas pyramidal}
Let $\phi$ be pyramidal. Then, $\phi$ is birational of type $(1,1,2)$ if and only if 
\begin{equation}
\label{eq: weight formulas pyramidal}
w_{ijk} = \alpha_i \, \beta_j \, \gamma_k \, \Delta_{ijk}
\end{equation}
for some $(\alpha_0:\alpha_1)\times (\beta_0:\beta_1)\times (\gamma_0:\gamma_1)$ in $(\P_\R^1)^3$. 
\end{corollary}

\subsubsection{The contractions of a pyramidal birational map}

\begin{notation}
Let $\phi$ be pyramidal with weights as \cref{eq: weight formulas pyramidal}. 
By \cref{corollary: weight formulas pyramidal}, $\phi$ is birational of type $(1,1,2)$, and by \cref{type pyramidal} the pencil of $u$-surfaces contains the following quadrics:
\vskip2pt
\begin{enumerate}
\item The rank-two quadric $\Pi_0\cup\Pi_1$, where $\Pi_1$ is the plane supporting the conic $C$. 
In particular, $\Pi_1$ is defined by 
$
\pi_1 = \bs{\pi}_1 \cdot \mbf{x}^T
$ 
for some $\bs{\pi}_1$ in $\R^4$
\item The cone (rank-three quadric) $Q$ with apex $\mbf{A} = s \cap t$ through the conic $C$. In particular, we find a $(\nu_0:\nu_1)$ in $\P_\R^1$ such that $Q$ is defined by 
$
q = \nu_0 \, \upsilon_0 + \nu_1 \, \upsilon_1 
$. 
\end{enumerate}  
\end{notation}

The rank-one condition of \cref{theorem: tensor pyramidal} readily determines all the contractions of $\phi$. 

\begin{corollary} 
\label{corollary: contractions pyramidal}
Let $\phi$ be pyramidal with weights as \cref{eq: weight formulas pyramidal}. 
Then, we have 
\begin{equation}
\label{eq: factorization second plane pyramidal}
\langle \bs{\pi}_1 , \mbf{f}\rangle 
= 
h(s_0,s_1,t_0,t_1) 
\,
c_0(u_0,u_1)
\end{equation}
for some linear $c_0 = c_0(u_0,u_1)$ and bilinear irreducible $h = h(s_0,s_1,t_0,t_1)$. 
Additionally, let    
\begin{equation}
a_0(s_0,s_1)
= 
\begin{vmatrix}
\phantom{-}s_0 & s_1 \\
-\alpha_1 & \alpha_0
\end{vmatrix}
\ ,\ 
b_0(t_0,t_1)
= 
\begin{vmatrix}
\phantom{-}t_0 & t_1 \\
-\beta_1 & \beta_0
\end{vmatrix}
\ ,\ 
c_1(u_0,u_1)
= 
\begin{vmatrix}
\phantom{-}u_0 & u_1 \\
-\gamma_1 & \gamma_0
\end{vmatrix}
\ .
\end{equation}
Then, we have the following:
\vskip2pt
\begin{enumerate}
\item $\phi$ contracts the surface $a_0 = 0$ (resp$.$ $b_0 = 0$) to the line $t$ (resp$.$ $s$)
\item $\phi$ contracts the surface $c_1 = 0$ to the point $\mbf{A} = s\cap t$ 
\item $\phi$ contracts the surface $h = 0$ to the plane conic $C$
\item The image of the surface $c_0 = 0$ is dense in $\Pi_1$
\end{enumerate}
\end{corollary}

\begin{proof}
By \cref{eq: weight formulas pyramidal}, the trilinear polynomial $\langle \bs{\pi}_0,\mbf{f}\rangle$ can be expanded as 
$$
\sum_{0\leq i,j,k\leq 1}
w_{ijk} 
\langle
\bs{\pi}_0 , \mbf{P}_{ijk}
\rangle
s_i t_j u_k
= 
\sum_{0\leq i,j,k\leq 1}
\left( 
\frac{ w_{ijk} }
{ \Delta_{ijk} } 
\right) 
s_i t_j u_k 
= 
\sum_{0\leq i,j,k\leq 1}
(\alpha_i s_i) (\beta_j t_j) (\gamma_k u_k)
= 
a_0 
b_0  
c_1
\ .
$$
By definition we have $\Sigma_i \cap \Pi_0 = \Sigma_0 \cap \Sigma_1 = s$, and the boundary $t$-lines $t_{i0},t_{i1}$ lie in $\Sigma_i$ for each $i = 0,1$. 
Hence, the four boundary $t$-lines intersect $\Pi_0$ at $s$. 
Therefore, as the image by $\phi$ of the surface $b_0 = 0$ lies in $\Pi_0$, it follows that $\phi$ contracts it to $s$. 
Similarly, it follows that $\phi$ contracts the surface $a_0=0$ to $t$. 
Additionally, the image of $c_1 = 0$ also lies in $\Pi_0$. 
Since the four boundary $u$-lines meet $\Pi_0$ at $\mbf{A}$, we conclude that 
$\phi$ contracts $c_1 = 0$ to $\mbf{A}$. 

On the other hand, by \cref{type pyramidal} the base locus of $\phi^{-1}$ is $Y_0\cap Y_1 = s\cup t\cup C$, hence defined by the ideal $(\upsilon_0 , \upsilon_1)\subset \R[x_0,x_1,x_2,x_3]$. 
Since $\phi$ is birational, we find 
$$
(\upsilon_0(\mbf{f}) , \upsilon_1(\mbf{f}))
= 
a_0 
b_0 
c_1 
h
\cdot 
(u_0 , u_1)
$$
for some $h = h(s_0,s_1,t_0,t_1)$ bilinear. 
In particular, the pullback ideal $(\upsilon_0(\mbf{f}) , \upsilon_1(\mbf{f}))\subset R$ defines the union of the surfaces $a_0 = 0$, $b_0 = 0$, $c_1 = 0$, and $h= 0$. 
Therefore, $\phi$ must contract $h=0$ to the conic $C$.  
Since $\Pi_1$ supports this conic, we find the factorization \cref{eq: factorization second plane pyramidal}.
In particular, the image of $c_0 = 0$ is dense in $\Pi_1$. 
\end{proof}

\subsection{Inverse and base locus}
\label{inverse: pyramidal}

In this subsection, we derive formulas for the inverse of a pyramidal birational map.
Additionally, we describe the base loci and blow-ups. 

\begin{theorem} 
\label{theorem: inverse pyramidal}
Let $\phi$ be pyramidal with weights as \cref{eq: weight formulas pyramidal}. 
Then, $\phi^{-1}$ is given by 
$$
(x_0 : x_1 : x_2 : x_3) 
\mapsto 
\left(
\alpha_1 \, \lambda_1 \, \sigma_1
: 
\alpha_0 \, \lambda_0 \, \sigma_0
\right)
\times
\left( 
\beta_1 \, \mu_1 \, \tau_1
: 
\beta_0 \, \mu_0 \, \tau_0
\right)
\times 
\left(
\gamma_1 \, \nu_1 \, \upsilon_1
: 
\gamma_0 \, \nu_0 \, \upsilon_0
\right)
\ .
$$
Additionally, with the notation of \cref{corollary: contractions pyramidal}, we have the following: 
\vskip2pt 
\begin{enumerate}
\item The base locus of $\phi$ is defined by the ideal of $R$ 
\begin{equation}
\label{saturation: pyramidal}
\left(
a_0
, 
b_0
, 
c_0
\right) 
\cap 
\left( 
h,
c_1
\right)
\end{equation}
\item The base locus of $\phi^{-1}$ is $Y_0\cap Y_1 = s\cup t\cup C$
\item $\phi$ blows up the base point $(a_0,b_0,c_0)$ to the plane $\Pi_0$
\item $\phi$ blows up the base curve $(h,c_1)$ to the cone $Q$
\end{enumerate}
\end{theorem}

\begin{proof}
By \cref{corollary: weight formulas pyramidal}, $\phi$ is birational of type $(1,1,2)$. 
In particular, the inverse on each factor of $(\P^1)^3$ is defined by a pencil of surfaces in $\P^3$ (recall \cref{trilinear birational}). 
As $\phi$ has type $(1,1,2)$, the inverses on the first two factors can be regarded as line isomorphisms $\sigma,\tau : \P^1 \xrightarrow{} L_s , L_t \subset (\P^3)^{\vee}$, where $(\P^3)^{\vee}$ is the dual space of $\P^3$, whereas the inverse on the third factor is a line isomorphism $\upsilon:\P^1 \xrightarrow{} L_u \subset  (\P^9)^{\vee}$, where $(\P^9)^{\vee}$ is the space of quadrics in $\P^3$. 
Moreover, by definition we have 
$$
\sigma:(1:0) \mapsto \sigma_0 
\ ,\ 
(0:1) \mapsto \sigma_1 
\ ,\ 
\tau:(1:0) \mapsto \tau_0 
\ ,\ 
(0:1) \mapsto \tau_1 
\ ,\
\upsilon: (1:0) \mapsto \upsilon_0 
\ ,\ 
(0:1) \mapsto \upsilon_1 
\ .
$$
Thus, each of these line isomorphisms is determined by the image of an additional point in $\P^1$. 
Since $\phi$ contracts $a_0 = 0$ and $b_0 = 0$ respectively  to the lines $t\subset \Pi_0$ and $s\subset \Pi_0$, we find 
$
\sigma 
(-\alpha_1:\alpha_0) = \pi_0 
$
and 
$
\tau
(-\beta_1:\beta_0) = \pi_0 
$. 
Therefore, the explicit formulas for $\sigma$ and $\tau$ are   
$$
(s_0:s_1) 
\mapsto 
\begin{vmatrix}
s_0 & s_1 \\ 
\alpha_1 \, \lambda_1 \, \sigma_1 & 
\alpha_0 \, \lambda_0 \,\sigma_0
\end{vmatrix}
= 
\begin{vmatrix}
s_0 & s_1 \\ 
A_0 & 
A_1
\end{vmatrix}
\ \ ,\ \ 
(t_0:t_1) 
\mapsto 
\begin{vmatrix}
t_0 & t_1 \\ 
\beta_1 \, \mu_1 \, \tau_1 & 
\beta_0 \, \mu_0 \,\tau_0
\end{vmatrix}
= 
\begin{vmatrix}
t_0 & t_1 \\ 
B_0 & 
B_1
\end{vmatrix}
\ .
$$
Before proving the formula for the third parameter, we prove the four claims in the statement.
First, the base locus of $\phi^{-1}$ follows immediately from \cref{type pyramidal}. 
Secondly, writing $I = 
(
a_0
, 
b_0
, 
c_0
)
\cap 
(
h,
c_1
)$ 
we have 
$f_0,f_1,f_2,f_3\in I_{(1,1,1)}$ by \cref{corollary: contractions pyramidal}. 
Since $\dim_\C I_{(1,1,1)} = 4$, it follows that $\C\langle f_0,f_1,f_2,f_3\rangle = I_{(1,1,1)}$.
Hence, $I$ defines the base locus of $\phi$.  
The last two statements follow from the proof of \cref{type pyramidal}. 
Finally, we find 
$
\upsilon 
(-\gamma_1:\gamma_0) 
=
q
$
and $\upsilon$ is given explicitly by 
$$
(u_0:u_1) 
\mapsto 
\begin{vmatrix}
u_0 & u_1 \\ 
\gamma_1 \, \nu_1 \, \upsilon_1 & 
\gamma_0 \, \nu_0 \,\upsilon_0
\end{vmatrix}
= 
\begin{vmatrix}
u_0 & u_1 \\ 
C_0 & 
C_1
\end{vmatrix}
\ ,
$$
yielding $\phi^{-1}$ as in the statement.  
\end{proof}

\section{Scaffold birational maps} 
\label{sec: scaffold}

In this section, we study the third family of birational maps: the class of scaffold birational maps. 

\begin{definition}
\label{scaffold}
A trilinear rational map is scaffold if 
all the boundary surfaces are quadrics except for one of the three parameters, for which:
\begin{enumerate}
\item The two boundary surfaces are planes, that intersect at a line $\ell$ 
\item The four boundary lines intersect two lines $r_0,r_1$, and these two intersect $\ell$
\end{enumerate}
\end{definition}

Given four pairwise skew lines in $\P^3$, not lying on a common quadric, there are exactly two secant lines 
that intersect the four of them. 
The two secants of the boundary lines are precisely $r_0$ and $r_1$.

\subsection{Geometric constraint}
\label{geometric constraint: scaffold}

As in \cref{geometric constraint: pyramidal}, we begin proving that a general birational map of type either $(1,2,2)$, $(2,1,2)$, or $(2,2,1)$ is scaffold. 
Once more, we restrict to type $(1,2,2)$. 

\begin{proposition} 
\label{type scaffold}
If $\phi$ is a general birational map of type $(1,2,2)$, then it is scaffold. Moreover,  
$$
T_0 \cap T_1 = s \cup r_0 \cup r_1 \cup y 
\ ,\ 
Y_0 \cap Y_1 = s \cup r_0 \cup r_1 \cup z 
\ ,
$$
where:
\vskip2pt
\begin{enumerate}
\item $s = \Sigma_0 \cap \Sigma_1$ is a line
\item $r_0,r_1$ are the unique two lines in $\P^3$ that meet $s_{jk}$ for every $0\leq j,k \leq 1$ 
\item $y$ is the line through the points $u_{00}\cap u_{01}$ and $u_{10}\cap u_{11}$
\item $z$ is the line through the points $t_{00}\cap t_{01}$ and $t_{10}\cap t_{11}$
\end{enumerate}
\end{proposition} 

\begin{proof}
First, since $\phi$ is general \cref{property: nondegeneracy} is satisfied. 
By \cite[Theorem $5.6$]{trilinear}, we find 
\begin{equation} 
(f_0,f_1,f_2,f_3)  : \mathfrak{N}^\infty  = (a_0 , b_0 , c_0) \cap (a_1 , b_1 , c_1) \cap (b_2 , c_2) 
\end{equation}
for some linear $a_i = a_i(s_0,s_1)$, $b_j = b_j(t_0,t_1)$, and $c_k = c_k(u_0,u_1)$.  
Additionally, $\phi^{-1}$ is defined by 
$$
(x_0,x_1,x_2,x_3) \mapsto (A_0:A_1)\times (B_0:B_1)\times (C_0:C_1)
$$
for some linear $A_i = A_i(\mbf{x})$ and quadratic $B_j = B_j(\mbf{x})$ and $C_k = C_k(\mbf{x})$. 
In particular, we find 
$$
(f_0(A_0,A_1,B_0,B_1,C_0,C_1) \, , \, \ldots \, , \, f_3(A_0,A_1,B_0,B_1,C_0,C_1)) 
= 
H
\cdot 
(x_0 , x_1 , x_2 , x_3)
$$ 
for some quartic form $H = H(\mbf{x})$. 
Therefore, the ideals in $\R[x_0,x_1,x_2,x_3]$
$$
(a_0(A_0,A_1),b_0(B_0,B_1),c_0(C_0,C_1)) 
,
(a_1(A_0,A_1),b_1(B_0,B_1),c_1(C_0,C_1)) 
,
(b_2(B_0,B_1),c_2(C_0,C_1))
$$ 
are principal and respectively generated by two linear and one quadratic forms. 
Thus, 
for each $i=0,1$ the polynomials $b_i(B_0,B_1)$ and $c_i(C_0,C_1)$ are divisible by $a_i(A_0,A_1)$, and we can write 
$$
b_i(B_0,B_1) = a_i(A_0,A_1) \, F_i 
\ ,\ 
c_i(C_0,C_1) = a_i(A_0,A_1) \, G_i 
$$  
for some linear forms $F_i = F_i(\mbf{x})$ and $G_i = G_i(\mbf{x})$. 
Therefore,  
\begin{gather*}
(B_0,B_1) = (b_0(B_0,B_1), b_1(B_0,B_1)) 
= 
( A_0 , A_1 ) 
\cap 
( a_0(A_0,A_1) , F_1 ) 
\cap 
( F_0 , a_1(A_0,A_1) ) 
\cap 
( F_0 , F_1 )
\ ,
\end{gather*} 
where we have used that $(A_0,A_1) = ( a_0(A_0,A_1) , a_1(A_0,A_1) ) $. 
Analogously,  
$$
(C_0,C_1) 
= 
( A_0 , A_1 ) 
\cap 
( a_0(A_0,A_1) , G_1 ) 
\cap 
( G_0 , a_1(A_0,A_1) ) 
\cap 
( G_0 , G_1 )
\ .
$$
Therefore, the two intersections $T_0\cap T_1$ and $Y_0\cap Y_1$ define unions of four lines. 
Since smooth quadrics are doubly ruled (i$.$e$.$ isomorphic to $\P^1\times\P^1$), and three skew lines determine a unique quadric, it follows that in each intersection there is a pair of lines in each ruling of the corresponding boundary surfaces. 
Moreover, by definition $\Sigma_0,\Sigma_1$ are planes and the ideal $(A_0,A_1)$ determines $s = \Sigma_0 \cap \Sigma_1$.  
Hence, the line $y$ (resp$.$ $z$) defined by $(F_0,F_1)$ (resp$.$ $(G_0,G_1)$) lies in the same ruling of $s$ in $T_j$ (resp$.$ $Y_k$). 
Additionally, for each $0 \leq i,j \leq 1$ the line $u_{ij}\subset T_j$ intersects $s$, since $u_{ij}\subset \Sigma_i$.  
Thus, $s$ and $u_{ij}$ belong to different rulings, and $y$ intersects $u_{ij}$ as well. 
Since $u_{i0},u_{i1}\subset \Sigma_i$ and $y\not\subset \Sigma_i$, $y$ must meet the point $u_{i0} \cap u_{i1}$. 
Therefore, $y$ is as claimed in the statement. 
With a parallel argument, we derive that $z$ is also as claimed. 
 
We conclude proving that the lines $r_0,r_1$ lie in both intersections.  
The line $s_{jk}\subset T_j$ (resp$.$ $s_{jk}\subset Y_k$) does not meet $s$ for any $0\leq j,k \leq 1$, so they lie in the same ruling. 
Therefore, the lines defined by 
$
( a_0(A_0,A_1) , F_1 ) 
$ and 
$
( F_0 , a_1(A_0,A_1) )
$ 
(resp$.$ 
$
( a_0(A_0,A_1) , G_1 ) 
$ and 
$
( G_0 , a_1(A_0,A_1) ) 
$)
must intersect the four boundary $s$-lines. 
Hence, they are the unique two with this property, namely $r_0$ and $r_1$. 
Since $r_0, r_1$ also  intersect $s$, it follows that $\phi$ is scaffold.
\end{proof}

In practice, the control points of a scaffold map can be generated with the following construction. 

\begin{construction}
\label{construction: scaffold}
The control points of a scaffold rational map can be generated as follows: 
\vskip2pt
\begin{enumerate}
\item Choose a line $s$ in $\P_\R^3$
\item Choose two skew lines $r_0,r_1$ in $\P_\R^3$ intersecting $s$
\item For each $0\leq j,k \leq 1$, choose pairwise distinct affine lines $s_{jk}$ meeting both $r_0,r_1$
\item Choose distinct planes $\Sigma_0,\Sigma_1$ through $s$, not containing $r_0,r_1$, such that $\Sigma_i \cap s_{jk}$ is affine
\item For each $0\leq i,j,k \leq 1$, define $\mbf{P}_{ijk} = \Sigma_i \cap s_{jk}$ 
\end{enumerate}
\end{construction} 

\subsection{Birational weights and tensor rank criterion}
\label{tensor criterion: scaffold}

In all the upcoming statements of \cref{sec: scaffold}, when $\phi$ is scaffold we assume that $\Sigma_0,\Sigma_1$ are the planar boundary surfaces. 

\begin{notation}
Let $\phi$ be scaffold, and maintain the notation of items $1$-$4$ in \cref{type scaffold}. 
Then, we find planes $\Pi_i, \Theta_j, P_k$ for each $0\leq i,j,k\leq 1$, respectively defined by 
$
\pi_i = \bs{\pi}_i \cdot \mbf{x}^T
$, 
$
\theta_j = \bs{\theta}_j \cdot \mbf{x}^T 
$, 
$
\rho_k = \bs{\rho}_k \cdot \mbf{x}^T
$ 
for some $\bs{\pi}_i, \bs{\theta}_j, \bs{\rho}_k$ in $\R^4$, such that: 
\vskip2pt
\begin{enumerate}
\item The plane $\Pi_i$ is defined by the lines $s$ and $r_i$. In particular, we find a $(\lambda_{0i}:\lambda_{1i})$ in $\P_\R^1$ such that 
$
\pi_i = \lambda_{0i} \, \sigma_0 + \lambda_{1i} \, \sigma_i
$
\item The plane $\Theta_j$ is defined by the lines $y$ and $r_j$
\item The plane $P_k$ is defined by the lines $z$ and $r_k$
\end{enumerate} 
\vskip2pt
\noindent Additionally, for each $l = 0,1$ we define 
$$ 
\Delta_{ijk}(l,y) = 
\frac{1}{
\langle \bs{\theta}_l , \mbf{P}_{ijk} \rangle
}
\ \ ,\ \ 
\Delta_{ijk}(l,z) = 
\frac{1}{
\langle \bs{\rho}_l , \mbf{P}_{ijk} \rangle
}
\ .
$$ 
\end{notation}

\begin{remark}
\label{remark: deltas scaffold}
For scaffold maps, both $\Delta_{ijk}(l,y)$ and $\Delta_{ijk}(l,z)$ are well-defined for every $0\leq i,j,k\leq 1$. 
Specifically, 
$\Delta_{ijk}(l,y)$ (resp$.$ $\Delta_{ijk}(l,z)$) is not well-defined if and only if $\mbf{P}_{ijk}$ lies on the line $r_l$.  
If this is the case, namely if $\mbf{P}_{ijk} = \Sigma_i \cap s_{jk}$ lies on $r_l$, it follows that $\Sigma_i = \text{Plane}(s,r_l)$. 
Hence, we can write 
$\mbf{P}_{ij'k'} = 
\Sigma_i \cap s_{j'k'} = 
\text{Plane}(s,r_l)
\cap
s_{j'k'}
= 
r_l
\cap
s_{j'k'} 
$ for every $0\leq j',k'\leq 1$. 
Therefore, $\mbf{P}_{i00},\mbf{P}_{i10},\mbf{P}_{i01}$, and $\mbf{P}_{i11}$ are aligned, and $\Sigma_i = r_l$ against \cref{property: nondegeneracy}.
\end{remark}

We now prove a characterization of birationality relying on tensor rank for scaffold maps. 

\begin{theorem} 
\label{theorem: tensor scaffold}
Let $\phi$ be scaffold. The following are equivalent:
\vskip2pt
\begin{enumerate}
\item $\phi$ is birational of type $(1,2,2)$
\item One of the four tensors 
$$
W(l,y) = \left( \frac{w_{ijk}}{\Delta_{ijk}(l,y)} \right)_{0\leq i,j,k \leq 1} \ \ , \ \ 
W(l,z) = \left( \frac{w_{ijk}}{\Delta_{ijk}(l,z)} \right)_{0\leq i,j,k \leq 1}
$$
for each $l = 0,1$, has rank one 
\item The four tensors above have rank one 
\end{enumerate}
\end{theorem} 

\begin{proof}
First, we prove that $2$ implies $1$. 
Without loss of generality, we assume that 
$
W(0,y) 
= 
(\alpha_{00},\alpha_{10}) \otimes (\beta_{00},\beta_{10}) \otimes (\gamma_{02},\gamma_{12})
$. 
In particular, we have the factorization 
$$
\langle
\bs{\theta}_0
, 
\mbf{f} 
\rangle
= 
\sum_{0\leq i,j,k \leq 1}
w_{ijk}
\langle
\bs{\theta}_0
, 
\mbf{P}_{ijk} 
\rangle
s_i t_j u_k
= 
\sum_{0\leq i,j,k \leq 1}
\left(
\frac{w_{ijk}}{\Delta_{ijk}}
\right)
s_i t_j u_k
= 
a_0 b_0 c_2
\ ,
$$
where 
\begin{equation}
\label{eq: a0 scaffold tensor}
a_0(s_0,s_1) = 
\begin{vmatrix}
\phantom{-}s_0 & s_1 \\
-\alpha_{10} & \alpha_{00}
\end{vmatrix}
\ \ ,\ \ 
b_0(t_0,t_1) = 
\begin{vmatrix}
\phantom{-}t_0 & t_1\\ 
-\beta_{10} & \beta_{00}
\end{vmatrix}
\ \ ,\ \ 
c_2(u_0,u_1) = 
\begin{vmatrix}
\phantom{-}u_0 & u_1\\
-\gamma_{12} & \gamma_{02}
\end{vmatrix}
\ .
\end{equation} 
Therefore, the images by $\phi$ of the surfaces defined by $a_0 = 0$, $b_0 = 0$, and $c_2 = 0$ lie in $\Theta_0$. 
We make the following observations:
\vskip2pt
\begin{enumerate}
\item As $\phi$ is scaffold, the four boundary $s$-lines meet $\Theta_0$ at $r_0$. 
Hence, $\phi$ contracts $a_0 = 0$ to $r_0$ 
\item Similarly, the four boundary $u$-lines meet $\Theta_0$ at the line $y$.
Hence, $\phi$ contracts $c_2 = 0$ to $y$
\end{enumerate}
\vskip2pt 
\noindent In particular, the line isomorphism $\sigma: \P^1\xrightarrow{}L_s\subset\left(\P^3\right)^\vee$ defined by
$
(1:0) \mapsto \sigma_0
$, 
$ 
(0:1) \mapsto \sigma_1
$, 
$
(-\alpha_{10}:\alpha_{00}) \mapsto \pi_0
$
yields a syzygy of degree $(1,0,0)$ of $\mbf{f}$, since $\sigma(s_0:s_1)(\mbf{f})$ reduces to a bivariate quadratic form that vanishes at three distinct points. 
Additionally, $\sigma$ parametrizes the pencil of planes spanned by $\Sigma_0,\Sigma_1$. 
Thus, as $r_1$ intersects $s$ the plane $\Pi_1$ lies in the image of $\sigma$. Hence, we find a point $(\alpha_{01}:\alpha_{11})$ in $\P_\R^1$ for which 
$\sigma(-\alpha_{11},\alpha_{01}) = \pi_1$.
By the same observation as before, $\phi$ contracts the surface defined by 
$
a_1(s_0,s_1) 
= 
\alpha_{01} s_0 
+ 
\alpha_{11} s_1
$
to $r_1$. 
Furthermore, as $\Theta_1$ contains the lines $y$ and $r_1$ we find the factorization 
$
\langle
\bs{\theta}_1
, 
\mbf{f} 
\rangle
= 
a_1 
b_1
c_2
$
where 
$
b_1(t_0,t_1) = 
\beta_{01} t_0 + \beta_{11} t_1
$
for some $(\beta_{01}:\beta_{11})$ in $\P_\R^1$. 
Therefore, by definition 
$
\bs{\theta}(s_0,s_1,t_0,t_1)
=
a_1
b_1
\,
\bs{\theta}_0
-
a_0
b_0
\,
\bs{\theta}_1
$
is a syzygy of $\mbf{f}$ of degree $(1,1,0)$. 
By hypothesis, $\mbf{f}$ does not admit syzygies of degrees $(0,1,0)$ and $(0,0,1)$, since $T_j$ and $Y_k$ are smooth quadrics for each $0\leq j,k\leq 1$. 
Hence, by \cite[Theorem $6.1$]{trilinear} $\phi$ is birational of type $(1,2,2)$.

We conclude proving that $1$ implies $3$, since the implication from $3$ to $2$ is immediate.  
If $\phi$ is birational of type $(1,2,2)$, by \cref{trilinear birational} $\mbf{f}$ admits a syzygy of degree $(1,0,0)$. 
Moreover, by \cref{lemma: linear syz s} this syzygy has the form 
$
\bs{\sigma}(s_0,s_1) = \alpha_0 \, \bs{\sigma}_0 \, s_0 + \alpha_1 \, \bs{\sigma}_1 \, s_1 
$. 
Since the lines $r_0,r_1$ intersect $s$, $\Pi_0$ and $\Pi_1$ lie in the pencil of $s$-planes. 
Therefore, for each $i=0,1$ we find a $a_i(s_0,s_1) = \alpha_{0i}s_0 + \alpha_{1i}s_1$ such that $\phi$ contracts $a_i = 0$ to $r_i$, as the image lies on $\Pi_i \cap \bs{\sigma}(-\alpha_{1i},\alpha_{0i}) = r_i$. 
Moreover, the pullback of $s$ is the surface defined by the bilinear polynomial $h = h(t_0,t_1,u_0,u_1)$ whose coefficients (in the monomial basis) are any of the two proportional rows of \cref{matrix: syz s}. 
On the other hand, by \cref{type scaffold} the ideal $(\tau_0,\tau_1)\subset \R[x_0,x_1,x_2,x_3]$ defines 
$
T_0 \cap T_1 = 
s \cup r_0 \cup r_1 \cup y
$. 
As $\phi$ is birational with a quadratic inverse on the second parameter, by the contractions explained we have  
$$
(\tau_0(\mbf{f}),\tau_1(\mbf{f}) )
= 
a_0 
a_1 
c_2
h 
\cdot 
(t_0,t_1) 
$$
for some linear $c_2 = c_2(u_0,u_1)$.
In particular, $\phi$ contracts $c_2 = 0$ to $y$ and we find the factorizations 
$
\langle 
\bs{\theta}_0 , \mbf{f}
\rangle 
= 
a_0 b_0 c_2
$,
$
\langle 
\bs{\theta}_1 , \mbf{f}
\rangle 
= 
a_1 b_1 c_2
$ 
for some linear $b_j = b_j(t_0,t_1)$, for each $j = 0,1$. 
Equivalently, the tensors $W(0,y)$, $W(1,y)$ have rank one. 
With a parallel argument, we conclude that $W(k,z)$ has rank one for each $k = 0,1$. 
\end{proof}
  
\begin{corollary}
\label{corollary: weight formulas scaffold}
Let $\phi$ be scaffold. 
Then, $\phi$ is birational of type $(1,2,2)$ if and only if one, and therefore all up to a nonzero scalar multiplying all the weights, of the following conditions holds:
\vskip2pt 
\begin{enumerate}
\item $w_{ijk} = \alpha_{i0} \, \beta_{j0} \, \gamma_{k2} \, \Delta_{ijk}(0,y)$ for each $0\leq i,j,k\leq 1$
\item $w_{ijk} = \alpha_{i1} \, \beta_{j1} \, \gamma_{k2} \, \Delta_{ijk}(1,y)$ for each $0\leq i,j,k\leq 1$
\item $w_{ijk} = \alpha_{i0} \, \beta_{j2} \, \gamma_{k0} \, \Delta_{ijk}(0,z)$ for each $0\leq i,j,k\leq 1$
\item $w_{ijk} = \alpha_{i1} \, \beta_{j2} \, \gamma_{k1} \, \Delta_{ijk}(1,z)$ for each $0\leq i,j,k\leq 1$
\end{enumerate}
\vskip2pt 
\noindent for some $(\alpha_{0i}:\alpha_{1i})$, $(\gamma_{0j}:\beta_{1j})$, $(\gamma_{0k}:\gamma_{1k})$ in $(\P_\R^1)^3$.
\end{corollary}

\subsubsection{The contractions of a scaffold birational map}

\begin{notation}
Let $\phi$ be scaffold with weights as in \cref{corollary: weight formulas scaffold}. 
In particular, $\phi$ is birational of type $(1,2,2)$, and by \cref{type scaffold} we have the following:
\vskip2pt
\begin{enumerate}
\item The rank-two quadrics $\Pi_0\cup\Theta_1$ and $\Pi_1\cup\Theta_0$ belong to the pencil of $t$-surfaces
\item The rank-two quadrics $\Pi_0\cup P_1$ and $\Pi_1\cup P_0$ belong to the pencil of $u$-surfaces
\item The unique smooth quadric $Q$ defined by the lines $s$, $y$, and $z$ belongs to the two pencils of $t$- and $u$-surfaces. In particular, $Q$ is defined by a quadratic form $q = q(\mbf{x})$, and we find points $(\mu_0:\mu_1)$ in $(\nu_0:\nu_1)$ in $\P_\R^1$ such that 
$ 
q = \mu_0 \, \tau_0 + \mu_1 \, \tau_1
=
\nu_0 \, \upsilon_0 + \nu_1 \, \upsilon_1
$
\end{enumerate}  
\end{notation}

As for the class of pyramidal maps, the rank-one conditions of \cref{theorem: tensor scaffold} determine the contractions of a scaffold birational map.

\begin{corollary}
\label{corollary: contractions scaffold}
Let $\phi$ be scaffold with weights as in \cref{corollary: weight formulas scaffold}, and let 
\begin{equation}
a_i(s_0,s_1)
= 
\begin{vmatrix}
\phantom{-}s_0 & s_1 \\
-\alpha_{1i} & \alpha_{0i}
\end{vmatrix}
\ ,\ 
b_j(t_0,t_1)
= 
\begin{vmatrix}
\phantom{-}t_0 & t_1 \\
-\beta_{1j} & \beta_{0j}
\end{vmatrix}
\ ,\ 
c_k(u_0,u_1)
= 
\begin{vmatrix}
\phantom{-}u_0 & u_1 \\
-\gamma_{1k} & \gamma_{0k}
\end{vmatrix}
\end{equation}
for $i=0,1$ and $0\leq j,k\leq 2$. 
Then, we have the following:
\vskip2pt
\begin{enumerate}
\item For each $i = 0,1$, $\phi$ contracts the surface $a_i = 0$ to the line $r_i$
\item $\phi$ contracts the surface $b_2 = 0$ to the line $z$ 
\item $\phi$ contracts the surface $c_2 = 0$ to the line $y$
\item For each $j = 0,1$, the image of the surface $b_j = 0$ is dense in the plane $\Theta_j$
\item For each $k = 0,1$, the image of the surface $c_k = 0$ is dense in the plane $P_k$
\end{enumerate}
\end{corollary}

\begin{proof}
Proof of $1$ implies $3$ in \cref{theorem: tensor scaffold}. 
\end{proof}

\subsection{Inverse and base locus}
\label{inverse: scaffold}

In analogy to \cref{inverse: hexahedral} and \cref{inverse: pyramidal}, we now derive the explicit formulas for the inverse of a scaffold birational map, and describe the base loci and blow-ups.

\begin{theorem} 
\label{theorem: inverse scaffold}
Let $\phi$ be scaffold with weights as in \cref{corollary: weight formulas scaffold}. 
Then, $\phi^{-1}$ is given by 
$$
(x_0 : x_1 : x_2 : x_3) 
\mapsto 
\left(
\alpha_{1i} \, \lambda_{1i} \, \sigma_1
: 
\alpha_{0i} \, \lambda_{0i} \, \sigma_0
\right)
\times
\left( 
\beta_{12} \, \mu_1 \, \tau_1
: 
\beta_{02} \, \mu_0 \, \tau_0
\right)
\times 
\left(
\gamma_{12} \, \nu_1 \, \upsilon_1
: 
\gamma_{02} \, \nu_0 \, \upsilon_0
\right)
\ ,
$$
where any $i = 0,1$ is valid. 
Additionally, with the notation of \cref{corollary: contractions scaffold}, we have the following: 
\vskip2pt 
\begin{enumerate}
\item The base locus of $\phi$ is defined by the ideal  
\begin{equation}
\label{saturation scaffold}
\left(
a_0
, 
b_0
, 
c_0
\right) 
\cap 
\left(
a_1
, 
b_1
, 
c_1
\right) 
\cap 
\left( 
b_2 
, 
c_2
\right)
\end{equation}
\item The base locus of $\phi^{-1}$ is $s\cup r_0 \cup r_1 \cup y\cup z$
\item For each $i = 0,1$, $\phi$ blows up the base point $(a_i,b_i,c_i)$ to the plane $\Pi_i$
\item $\phi$ blows up the base line $(b_2,c_2)$ to the smooth quadric $Q$
\end{enumerate}
\end{theorem}

\begin{proof}
By \cref{corollary: weight formulas scaffold}, $\phi$ is birational of type $(1,2,2)$. 
As explained in the proof of \cref{theorem: inverse pyramidal}, the inverses on each factor of $\PPP$ can be regarded as line isomorphisms $\sigma: \P^1 \xrightarrow{} L_s\subset (\P^3)^\vee$ and $\tau,\upsilon:\P^1 \xrightarrow{} L_t,L_u\subset (\P^9)^\vee$, and are determined by the image of a point distinct from $(1:0)$ and $(0:1)$.   
By \cref{corollary: contractions scaffold}, $\phi$ contracts $a_i = 0$ to $r_i\subset \Pi_i$ for each $i = 0,1$. 
Hence,
$
\sigma 
(-\alpha_{1i},\alpha_{0i}) 
=
\pi_i
$
yielding the inverse for the first parameter. 
Similarly, by \cref{corollary: contractions scaffold} $\phi$ contracts $b_2 = 0$ (resp$.$ $c_2 = 0$) to $z\subset Q$ (resp$.$ $y\subset Q$). 
Since $Q\equiv q = 0$ is the unique $t$-surface (resp$.$ $u$-surface) containing the line $z$ (resp$.$ $y$) we find 
$
\tau
( 
-\beta_{12},\beta_{02} 
)
=
\upsilon 
(
-\gamma_{12},\gamma_{02}
)
=
q
$, 
and the inverse follows. 

Regarding the four claims in the statement, the base locus of $\phi^{-1}$ and the blow-ups follow immediately from \cref{type scaffold}. 
Finally, writing $I = (a_0,b_0,c_0)\cap (a_1,b_1,c_1)\cap (b_2,c_2)$ we find 
$$
I_{(1,1,1)} 
= 
\R
\langle 
a_0 b_0 c_2
,
a_1 b_1 c_2
, 
a_0 b_2 c_0
,
a_1 b_2 c_1
\rangle 
= 
\R
\langle 
\langle 
\bs{\theta}_0
, 
\mbf{f}
\rangle
,
\langle 
\bs{\theta}_1
, 
\mbf{f}
\rangle
,
\langle 
\bs{\rho}_0
, 
\mbf{f}
\rangle
, 
\langle 
\bs{\rho}_1
, 
\mbf{f}
\rangle
\rangle 
= 
\R\langle f_0,f_1,f_2,f_3\rangle
\ ,
$$
and $I$ defines the base locus of $\phi$. 
\end{proof}

\section{Tripod birational maps}
\label{sec: tripod}

We conclude addressing the last, and geometrically most complex, family of birational maps: the class of tripod birational maps. 

\begin{definition}
\label{tripod}
A trilinear rational map is tripod if all the boundary surfaces are quadrics and: 
\begin{enumerate}
\item The four boundary $s$-lines (resp$.$ $t$-lines and $u$-lines) intersect a line $s$ (resp$.$ $t$ and $u$) 
\item The three lines $s,t,u$ intersect at a point $\mbf{A}$
\item All the boundary lines intersect a plane conic $C$ which intersects the three lines $s,t$, and $u$
\end{enumerate}
\end{definition}

\subsection{Geometric constraint} 

As in \cref{geometric constraint: pyramidal} and \cref{geometric constraint: scaffold}, we begin proving that a general birational map of type $(2,2,2)$ is tripod. 

\begin{proposition} 
\label{type tripod}
If $\phi$ is a general birational map of type $(2,2,2)$, then it is tripod. 
Moreover, 
$$
\Sigma_0 \cap \Sigma_1 = t \cup u \cup C 
\ ,\ 
T_0 \cap T_1 = s \cup u \cup C 
\ ,\ 
Y_0 \cap Y_1 = s \cup t \cup C 
\ .
$$
\end{proposition} 

\begin{proof}
Since $\phi$ is general  \cref{property: nondegeneracy} is satisfied, and by \cite[Theorem $5.8$]{trilinear} we have
\begin{equation}
\label{eq: (222) BI} 
(f_0,f_1,f_2,f_3) : \mathfrak{N}^\infty  = (a_0 , b_0 , c_0) \cap (a_1^2 , b_1^2 , c_1^2 , a_1 b_1 , a_1 c_1 , b_1 c_1 , a_1 b_0 c_0 + a_0 b_1 c_0 + a_0 b_0 c_1 )
\end{equation}
for some linear $a_i = a_i(s_0,s_1)$, $b_j = b_j(t_0,t_1)$, $c_k = c_k(u_0,u_1)$. 
With the notation of \cref{eq: inverse map} for $\phi^{-1}$, 
$$
(f_0(A_0,A_1,B_0,B_1,C_0,C_1) \, , \, \ldots \, , \, f_3(A_0,A_1,B_0,B_1,C_0,C_1)) 
= 
F
\cdot 
(x_0 , 
x_1 , 
x_2 , 
x_3 )
$$ 
for some quintic form  $F = F(\mbf{x})$. 
In particular, $F=0$ defines the projection to $\P^3$ of the exceptional divisor $E$ of the blow-up $\Gamma\subset (\P^1)^3\times\P^3$ of $(\P^1)^3$ along the base locus of $\phi$. 
By \cref{eq: (222) BI},  $E$ has two irreducible components. 

Let 
$\Tilde{a}_i = a_i(A_0,A_1)$, $\Tilde{b}_j = b_j(B_0,B_1)$, and $\Tilde{c}_k = c_k(C_0,C_1)$ 
for each $0\leq i,j,k\leq 1$. 
Then, the components of $E$ project to the surfaces in $\P^3$ defined by $G_0 = 0$ and $G_1 = 0$, where 
$$
G_0 = 
\gcd 
\left(
\Tilde{a}_0, \Tilde{b}_0 , \Tilde{c}_0
\right)
\ ,\ 
G_1 = 
\gcd 
\left(
{\Tilde{a}_1}^2, \Tilde{b}_1^2 , \Tilde{c}_1^2, 
\Tilde{a}_1 \Tilde{b}_1, 
\Tilde{a}_1 \Tilde{c}_1, 
\Tilde{b}_1 \Tilde{c}_1, 
\Tilde{a}_1 \Tilde{b}_0 \Tilde{c}_0 + 
\Tilde{a}_0 \Tilde{b}_1 \Tilde{c}_0 + 
\Tilde{a}_0 \Tilde{b}_0 \Tilde{c}_1
\right)
$$ 
and $F = G_0 G_1$. 
In particular, since 
$
1 \leq \deg G_0 \leq 2
$ 
it follows that $\deg G_1\geq 3$. 
Therefore, the degree of 
$
G' = 
\gcd
(
\Tilde{a}_1^2 , 
\Tilde{b}_1^2 , 
\Tilde{c}_1^2 , 
\Tilde{a}_1 \Tilde{b}_1 , 
\Tilde{a}_1 \Tilde{c}_1 , 
\Tilde{b}_1 \Tilde{c}_1 
)
$
is also at least three. 
It follows that $\deg G' = 4$, and $\Tilde{a}_1,\Tilde{b}_1,\Tilde{c}_1$ are proportional. 
On the other hand, if $G_0$ is quadratic then $\Tilde{a}_0,\Tilde{b}_0,\Tilde{c}_0$ are also proportional, implying 
$$
\Tilde{a}_1 \Tilde{b}_0 \Tilde{c}_0 + \Tilde{a}_0 \Tilde{b}_1 \Tilde{c}_0 + \Tilde{a}_0 \Tilde{b}_0 \Tilde{c}_1
= 
\lambda_1
\, 
\Tilde{a}_1 \Tilde{a}_0^2 
=
\lambda_2
\, 
\Tilde{b}_1 \Tilde{b}_0^2
= 
\lambda_3
\, 
\Tilde{c}_1 \Tilde{c}_0^2
$$
for some nonzero $\lambda_1,\lambda_2,\lambda_3$. 
Since $\gcd(A_0,A_1) = 1$ we find 
$
G_1 = \gcd(\Tilde{a}_1^2,\Tilde{a}_1 \Tilde{a}_0^2) 
= 
\Tilde{a}_1
$, yielding a contradiction. 
Hence, we must have $(\deg G_0,\deg G_1) = (1,4)$. 

Therefore, $\Tilde{a}_1,\Tilde{b}_1$, and $\Tilde{c}_1$ are proportional to a quadratic form $q = q(\mbf{x})$, we have $G_1 = q^2$, and moreover 
$
\Tilde{a}_0 = \pi_1 \pi
$, 
$
\Tilde{b}_0 = \pi_2 \pi 
$, 
$ 
\Tilde{c}_0 = \pi_3 \pi
$
for some linear forms $\pi = \pi(\mbf{x})$ and $\pi_r = \pi_r(\mbf{x})$ for $r = 1,2,3$.  
In particular, we can write 
$$
(A_0,A_1) 
= 
(\Tilde{a}_0,\Tilde{a}_1)
= 
(\pi_1 \pi, q)
= 
(\pi_1, q) \cap (\pi, q)
$$
and 
$
\Tilde{a}_1 \Tilde{b}_0 \Tilde{c}_0 + \Tilde{a}_0 \Tilde{b}_1 \Tilde{c}_0 + \Tilde{a}_0 \Tilde{b}_0 \Tilde{c}_1 
= 
\pi^2 
q
\, 
\left(
\omega_1 \, \pi_2 \pi_3 + 
\omega_2 \, \pi_1 \pi_3 + 
\omega_3 \, \pi_1 \pi_2
\right)
$
for some nonzero $\omega_1,\omega_2,\omega_3$. 
Since the polynomial above is divisible by $G_1 = q^2$, and $q$ and $\pi$ are coprime, it follows that 
$
\omega_1 \, \pi_2 \pi_3 + 
\omega_2 \, \pi_1 \pi_3 + 
\omega_3 \, \pi_1 \pi_2
$ 
is proportional to $q$. 
Hence,  
$$
(\pi_1,q)
= 
(\pi_1, \pi_2 \pi_3)
= 
(\pi_1, \pi_2) \cap (\pi_1, \pi_3)
\ . 
$$
With parallel arguments, we derive 
$$
(B_0,B_1) 
=
(\pi_2,\pi_1) 
\cap
(\pi_2,\pi_3) 
\cap
(\pi,q)
\ ,\ 
(C_0,C_1) 
=
(\pi_3,\pi_1) 
\cap
(\pi_3,\pi_2) 
\cap
(\pi,q)
\ \ .
$$
Let $(\pi_2,\pi_3), (\pi_1,\pi_3)$, and $(\pi_1,\pi_2)$ respectively define lines $s,t$, and $u$, and let $(\pi,q)$ define a plane conic $C$. 
By definition, $s\cap t\cap u = \mbf{A}$ and the four boundary $s$-lines (resp$.$ $t$- and $u$-lines) intersect $s$ (resp$.$ $t$ and $u$) since they lie in the base loci $T_0\cap T_1$ and $Y_0\cap Y_1$ (resp$.$ in $\Sigma_0\cap \Sigma_1$, $Y_0\cap Y_1$ and $\Sigma_0\cap \Sigma_1$, $T_0\cap T_1$). 
Additionally, $C$ intersects $s,t$, and $u$. 
Moreover, $C$ is intersected by all the boundary lines, since it lies in the base locus of all the parametric surfaces. 
Therefore, $\phi$ is tripod. 
\end{proof}

Let $\ell$ and $C$ be respectively a line and a plane conic in $\P^3$, such that $\ell\cap C = \mbf{Q}_1$ is a point. 
Additionally, let $\mbf{P}$ be a general point in $\P^3$. 
Then, the plane through $\mbf{P}$ and $\ell$ intersects $C$ at exactly two points $\mbf{Q}_1,\mbf{Q}_2$. 
In particular, $\ovl{\mbf{P}\mbf{Q}_2}$ is the unique line through $\mbf{P}$ that intersects $\ell$ and $C$ at distinct points. 
In practice, the control points of a tripod map can be generated as follows. 

\begin{construction}
\label{construction: tripod}
The control points of a tripod rational map can be generated as follows: 
\vskip2pt
\begin{enumerate}
\item Choose a point $\mbf{A}$ in $\P_\R^3$ 
\item Choose three distinct lines $s,t,u$ in $\P_\R^3$ through $\mbf{A}$ 
\item Choose a plane $\Pi\subset \P_\R^3$ not containing any of the lines $s,t,u$
\item Choose a plane conic $C\subset \Pi$ through $s\cap \Pi$, $t\cap \Pi$, and $u\cap \Pi$
\item Choose an affine control point $\mbf{P}_{000}$ not lying on $s,t,u$ or $\Pi$
\item Define $s_{00}$ (resp$.$ $t_{00},u_{00}$) as the line through $\mbf{P}_{000}$ that intersects $s$ (resp$.$ $t,u$) and $C$ at distinct points
\item Choose affine control points $\mbf{P}_{100}$, $\mbf{P}_{010},\mbf{P}_{001}$, respectively on $s_{00},t_{00},u_{00}$, all distinct and not lying on $s,t,u$
\item Define $t_{10},u_{10}$ (resp$.$ $s_{10},u_{01}$ and $s_{01},t_{01}$) as in $6$ from the point $\mbf{P}_{100}$ (resp$.$ $\mbf{P}_{010}$ and $\mbf{P}_{001}$)
\item Define $\mbf{P}_{011} = t_{01}\cap u_{01}$, $\mbf{P}_{101} = s_{01}\cap u_{10}$, and $\mbf{P}_{110} = s_{10}\cap t_{10}$
\item Define $s_{11}, t_{11}, u_{11}$ as in $6$ respectively from the points $\mbf{P}_{011},\mbf{P}_{101}, \mbf{P}_{110}$
\item Define $\mbf{P}_{111} = s_{11}\cap t_{11}\cap u_{11}$
\end{enumerate}
\end{construction}

In \cref{example: quadcube} we illustrate how  \cref{construction: tripod} can be used in practice with an example. 

\subsection{Birational weights and tensor rank criterion} 
\label{tensor criterion: tripod}

\begin{notation}
Let $\phi$ be tripod. 
We denote by $\Pi_1$ (resp$.$ $\Pi_2$ and $\Pi_3$) the plane through $t,u$ (resp$.$ $s,u$ and $s,t$), which is defined by 
$\pi_r = \bs{\pi}_r \cdot \mbf{x}^T
$ for $r=1$ (resp$.$ $r=2$ and $r=3$) and some $\bs{\pi}_r$ in $\R^4$. 
Additionally, for each $0\leq i,j,k\leq 1$, we define 
$$ 
\Delta_{ijk}(r) = 
\frac{1}{
\langle \bs{\pi}_r , \mbf{P}_{ijk} \rangle
}
\ .
$$ 
\end{notation}

\begin{remark}
For tripod maps, $\Delta_{ijk}(r)$ is  well-defined for every $0\leq i,j,k\leq 1$ and $r = 1,2,3$. 
Specifically, 
$\Delta_{ijk}(r)$ is not well-defined if and only if $\mbf{P}_{ijk}$ lies on the plane $\Pi_r$. 
Without loss of generality, let $r = 1$. 
Then, the boundary line $t_{ik}$ (resp$.$ $u_{ij}$) also lies in $\Pi_1$, since it intersects $t$ (resp$.$ $u$) and $\Pi_1 = \text{Plane}(t,u)$. 
Similarly, the control point $\mbf{P}_{ij^*k}$ (resp$.$ $\mbf{P}_{ijk^*}$) lies in $\Pi_1$, and the line $u_{ij^*}$ (resp$.$ $t_{ik^*}$) as well for the same reasons as before. 
Therefore, the four control points $\mbf{P}_{i00}, \mbf{P}_{i10}, \mbf{P}_{i01}$, and $\mbf{P}_{i11}$ are coplanar, yielding a contradiction with the fact that $\Sigma_i$ is a smooth quadric.
\end{remark}


\begin{theorem} 
\label{theorem: tensor tripod}
Let $\phi$ be tripod. The following are equivalent:
\vskip2pt
\begin{enumerate}
\item $\phi$ is birational of type $(2,2,2)$
\item One of the three tensors 
$$
W(r) = \left( \frac{w_{ijk}}{\Delta_{ijk}(r)} \right)_{0\leq i,j,k \leq 1} 
$$
for each $r = 1,2,3$, has rank one 
\item The three tensors above have rank one 
\end{enumerate}
\end{theorem} 

\begin{proof}
First, we prove that 1 implies both 2 and 3. 
For each $r=1,2,3$, the tensor $W(r)$ encodes the coefficients of the trilinear polynomial 
$
\langle \bs{\pi}_r , \mbf{f}\rangle
$,  
which defines the pullback by $\phi$ of $\Pi_r$. 
By \cref{type tripod}, the lines $s,t,u$ are irreducible components of the base locus of $\phi^{-1}$. 
Hence, $\phi$ contracts some surfaces in $(\P^1)^3$ to each of them. 
Since each plane $\Pi_1,\Pi_2,\Pi_3$ contains two of these lines, $\langle \bs{\pi}_r , \mbf{f} \rangle$ admits two factors that must be linear. Equivalently, $W(r)$ has rank one. 

We continue proving that $2$ implies $3$. 
Without loss of generality, we assume that
$
W(1) = 
(\alpha_{00},\alpha_{10}) \otimes (\beta_{01},\beta_{11}) \otimes (\gamma_{01},\gamma_{11})
$, yielding the factorization 
$
\langle
\bs{\pi}_1
, 
\mbf{f} 
\rangle
= 
a_0 b_1 c_1
$
where 
\begin{equation}
\label{eq: a0 tripod tensor}
a_0(s_0,s_1) = 
\begin{vmatrix}
\phantom{-}s_0 & s_1 \\
-\alpha_{10} & \alpha_{00}
\end{vmatrix}
\ \ ,\ \ 
b_1(t_0,t_1) = 
\begin{vmatrix}
\phantom{-}t_0 & t_1\\ 
-\beta_{11} & \beta_{01}
\end{vmatrix}
\ \ ,\ \ 
c_1(u_0,u_1) = 
\begin{vmatrix}
\phantom{-}u_0 & u_1\\
-\gamma_{11} & \gamma_{01}
\end{vmatrix}
\ .
\end{equation} 
Therefore, the images of $a_0 = 0$, $b_1 = 0$, and $c_1 = 0$ lie in $\Pi_1$. 
We make the following observations:
\vskip2pt
\begin{enumerate}
\item Since $\phi$ is tripod, the four boundary $t$-lines (resp$.$ $u$-lines) meet $\Pi_1$ at the line $t$ (resp$.$ $u$). 
Hence, $\phi$ contracts $b_1 = 0$ (resp$.$ $c_1 = 0$) to $t$ (resp$.$ $u$) 
\item The image of $a_0 = 0$ is dense in $\Pi_1$
\end{enumerate}
\vskip2pt 
\noindent Now, for each $0\leq i,j,k\leq 1$ consider the restrictions $\phi_{jk}:\P^1\xrightarrow{}s_{jk}\subset \P^3$ defined by 
$
\mbf{P}_{jk}(s_0,s_1) 
= 
w_{0jk} \, \mbf{P}_{0jk} \, s_0 
+ 
w_{1jk} \, \mbf{P}_{1jk} \, s_1
$,  
as well as the planes $T_{ik}\equiv \tau_{ik} = 0$ (resp$.$ $Y_{ij} \equiv \upsilon_{ij} = 0$) through $t,t_{ik}$ (resp$.$ $u,u_{ij}$), 
for some linear forms $\tau_{ik} = \tau_{ik}(\mbf{x})$ and $\upsilon_{ij} = \upsilon_{ij}(\mbf{x})$. 
By the second observation above, $\phi_{jk}$ is the unique line isomorphism sending 
$
(1:0) \mapsto \mbf{P}_{0jk}
$, 
$
(0:1) \mapsto \mbf{P}_{1jk}
$, 
$
(-\alpha_{10}:\alpha_{00}) \mapsto s_{jk}\cap \Pi_1
$.
Additionally, define the line isomorphisms
\begin{eqnarray*}
\theta_k: (s_0:s_1)\in\P^1 \xrightarrow{} (\P^3)^\vee
\ :\ &
(1:0) \mapsto \tau_{0k}
\ ,\ 
(0:1) \mapsto \tau_{1k} 
\ ,\ 
(-\alpha_{10}:\alpha_{00}) \mapsto \pi_1
\ ,
\\
\rho_j: (s_0:s_1)\in\P^1 \xrightarrow{} (\P^3)^\vee
\ :\ &
(1:0) \mapsto \upsilon_{0j}
\ ,\ 
(0:1) \mapsto \upsilon_{1j} 
\ ,\ 
(-\alpha_{10}:\alpha_{00}) \mapsto \pi_1 
\ .
\end{eqnarray*}
In particular, $\theta_k$ and $\rho_j$ parametrize the pencils of planes respectively defined by the lines $t$ and $u$. 
Let $\bs{\theta}_k = \bs{\theta}_k(s_0,s_1)$ and $\bs{\rho}_j = \bs{\rho}_j(s_0,s_1)$ be  tuples in $(\R[s_0,s_1]_1)^4$ of linear polynomials respectively defining these two  parametrizations. 
Hence,   
$
\langle
\bs{\theta}_k
,
\mbf{P}_{jk}
\rangle 
= 
\langle
\bs{\rho}_j
,
\mbf{P}_{jk}
\rangle 
= 
0 
$  
yielding independent syzygies of $\Z^3$-degree $(1,0,0)$ of $\mbf{P}_{jk}$. 
In particular, since $\Pi_2 = \text{Plane}(s,u)$ and $\Pi_3=\text{Plane}(s,t)$, the following conditions are equivalent:
\vskip2pt
\begin{enumerate}
\item $\mbf{P}_{jk}(s_0,s_1) = s_{jk} \cap s = s_{jk} \cap \Pi_2 = s_{jk} \cap \Pi_3$
\item $\bs{\theta}_k(s_0,s_1) = \bs{\pi}_3$ 
\item $\bs{\rho}_j(s_0,s_1) = \bs{\pi}_2 $
\end{enumerate}
\vskip4pt
\noindent  
Now, let $p = (-\alpha_{11}:\alpha_{01})\in\P^1$ 
such that $\mbf{P}_{00}(p) = s_{00} \cap s$. 
By the equivalences above, we find $\theta_0(p) = \pi_3$ and $\rho_0(p) = \pi_2$, 
which readily imply  
$
\mbf{P}_{10}(p) = s_{10} \cap s 
$ and 
$
\mbf{P}_{01}(p) = s_{01} \cap s
$. 
Again, these imply  
$\theta_1(p) = \pi_3$ and $\rho_1(p) = \pi_2$, hence
$\mbf{P}_{11}(p) = s_{11} \cap s$. 
Therefore, $\phi$ contracts the surface defined by $a_1(s_0,s_1) = \alpha_{01} s_0 + \alpha_{11} s_1$ to $s$.  
In particular, we can write    
\begin{equation}
\label{eq: factorizations 23 tripod}
\langle
\bs{\pi}_2 , \mbf{f} 
\rangle
= 
a_1 b_0 c_1 
\ ,\ 
\langle
\bs{\pi}_3 , \mbf{f} 
\rangle
= 
a_1 b_1 c_0 
\ ,
\end{equation}
for some linear $b_0 = b_0(t_0,t_1),c_0=c_0(u_0,u_1)$, 
and 3 follows. 

We conclude proving that 3 implies 1. 
Since $\phi$ is tripod, the cone $Q$ and the rank-two quadric $\Pi_1\cup\Pi$ (resp$.$ $\Pi_2\cup\Pi$ and $\Pi_3\cup\Pi$) lie in the pencil of surfaces spanned by $\Sigma_0,\Sigma_1$ (resp$.$ $T_0,T_1$ and $Y_0,Y_1$). 
In particular, the image of the line isomorphism 
$\varphi_1: \P^1\xrightarrow{}\left(\P^9\right)^\vee \simeq \P(\C[x_0,x_1,x_2,x_3]_2)$ defined by
$$
(\zeta_0:\zeta_1) \mapsto 
\begin{vmatrix}
\zeta_0 & \zeta_1 \\ 
\pi_1 \pi & q
\end{vmatrix}
$$
is the line $L_s$ spanned by $\sigma_0,\sigma_1$. 
By the same arguments as before, the rank-one conditions readily imply the contraction of $a_1 = 0$ (resp$.$ $b_1 = 0$ and $c_1=0$) to $s$ (resp$.$ $t$ and $u$),  and the image of $a_0 = 0$ being dense on $\Pi_1$. 
Hence, we find the factorizations 
$
\pi_1 (f_0,f_1,f_2,f_3)
= 
\langle
\bs{\pi}_1
, 
\mbf{f}
\rangle
= 
a_0 b_1 c_1
$ and 
$
q(f_0,f_1,f_2,f_3)
= 
a_1 b_1 c_1 h_1
$
for some trilinear $h_1 = h_1(s_0,s_1,t_0,t_1,u_0,u_1)$.
On the other hand, $\varphi_1$ induces the specialization 
$$
\varphi_1':
(\zeta_0:\zeta_1) \mapsto 
\varphi_1(\zeta_0:\zeta_1)(\mbf{f})
=
\begin{vmatrix}
\zeta_0 & \zeta_1 \\ 
a_0 b_1 c_1 h_0 & a_1 b_1 c_1 h_1
\end{vmatrix}
= 
 b_1 c_1
\begin{vmatrix}
\zeta_0 & \zeta_1 \\ 
a_0 h_0 & a_1 h_1
\end{vmatrix}
$$
where $h_0 = \pi(f_0,f_1,f_2,f_3) = \langle \bs{\pi} , \mbf{f} \rangle$. 
In particular, $L_1$ can be realized in the projective space of homogeneous polynomials of $\Z^3$-degree $(2,1,1)$, namely as the line spanned by $a_0h_0$ and $a_1h_1$. 
In this space, the locus of polynomials admitting a factor of $\Z^3$-degree $(1,0,0)$ 
is cut out by quadratic equations. 
Therefore, $L_1$ is either contained in this locus or intersects it in at most two points. 
However, by definition $L_1$ contains the four distinct polynomials 
$$
a_0 h_0 
\ ,\ 
a_1 h_1
\ ,\ 
b_1^{-1} c_1^{-1} \sigma_0(f_0,f_1,f_2,f_3) 
= 
s_0  h_0' 
\ ,\ 
b_1^{-1} c_1^{-1} \sigma_1(f_0,f_1,f_2,f_3) 
= 
s_1  h_1' 
\ , 
$$
for some trilinear $h_i' = h_i'(s_0,s_1,t_0,t_1,u_0,u_1)$, all admitting a factor of degree $(1,0,0)$. 
Hence, $L_1$ is the line of polynomials of the form $a(s_0,s_1)\cdot h$, where $h$ is proportional to $h_0$ and $h_1$.  
In particular, the rational map $\phi_1: \P^3 \dashrightarrow \P^1$ defined by 
$$
(x_0:x_1:x_2:x_3)
\mapsto 
(A_0:A_1)
=
(
\alpha_{11} \, \pi_1 \pi - \alpha_{10} \, q
:
-\alpha_{01} \, \pi_1 \pi + \alpha_{00} \, q 
)
$$
yields the inverse of $\phi$ on the first $\P^1$ factor of $(\P^1)^3$ since 
$
(A_0,A_1)
\xmapsto{x_i\mapsto f_i}
b_1 
c_1
h 
\, 
(s_0,s_1)
$.
With parallel arguments, we derive the inverse for the  remaining two factors of $(\P^1)^3$. 
\end{proof}

\begin{corollary}
\label{corollary: weight formulas tripod}
Let $\phi$ be tripod. 
Then, $\phi$ is birational of type $(2,2,2)$ if and only if one, and therefore all up to a nonzero scalar multiplying all the weights, of the following conditions holds:
\vskip2pt 
\begin{enumerate}
\item $w_{ijk} = \alpha_{i0} \, \beta_{j1} \, \gamma_{k1} \, \Delta_{ijk}(1)$ for each $0\leq i,j,k\leq 1$
\item $w_{ijk} = \alpha_{i1} \, \beta_{j0} \, \gamma_{k1} \, \Delta_{ijk}(2)$ for each $0\leq i,j,k\leq 1$
\item $w_{ijk} = \alpha_{i1} \, \beta_{j1} \, \gamma_{k0} \, \Delta_{ijk}(3)$ for each $0\leq i,j,k\leq 1$
\end{enumerate}
\vskip2pt 
\noindent for some $(\alpha_{0i}:\alpha_{1i})$, $(\gamma_{0j}:\beta_{1j})$, $(\gamma_{0k}:\gamma_{1k})$ in $\P_\R^1$.
\end{corollary}

\subsubsection{The contractions of a tripod birational map}

\begin{notation}
Let $\phi$ be tripod with weights as in \cref{corollary: weight formulas tripod}. 
In particular, $\phi$ is birational of type $(2,2,2)$, and by \cref{type tripod} we have the following:
\vskip2pt
\begin{enumerate}
\item The rank-two quadric $\Pi\cup\Pi_1$ (resp$.$ $\Pi\cup\Pi_2$ and $\Pi\cup\Pi_3$) belongs to the pencil of $s$-surfaces (resp$.$ $t$- and $u$-surfaces). 
Namely, we find $(\lambda_{00}:\lambda_{10})$ (resp$.$ $(\mu_{00}:\mu_{10})$ and $(\nu_{00}:\nu_{10})$) in $\P_\R^1$ such that $\Pi\cup\Pi_1$ (resp$.$ $\Pi\cup\Pi_2$ and $\Pi\cup\Pi_3$) is defined by 
$
\lambda_{00}\, \sigma_0 + \lambda_{10}\, \sigma_1 
= 0
$
(resp$.$ $\mu_{00}\, \tau_0 + \mu_{10}\, \tau_1 = 0
$
and 
$
\nu_{00}\, \upsilon_0 + \nu_{10}\, \upsilon_1 
= 0
$)
\item The cone $Q$ of apex $\mbf{A}$ through the plane conic $C$ belongs to the pencils of $s$-, $t$-, and $u$-surfaces. 
In particular, $Q$ is defined by a quadratic form $q = q(\mbf{x})$, and we find points $(\lambda_{01}:\lambda_{11}), (\mu_{01}:\mu_{11})$, and $(\nu_{01}:\nu_{11})$ in $\P_\R^1$ such that 
$
q = 
\lambda_{01}\, \sigma_0 + \lambda_{11}\, \sigma_1 
= 
\mu_{01}\, \tau_0 + \mu_{11}\, \tau_1 
= 
\nu_{01}\, \upsilon_0 + \nu_{11}\, \upsilon_1
$
\end{enumerate}  
\end{notation}

As for the previous classes of birational maps, 
\cref{theorem: tensor tripod} readily implies the following corollary. 

\begin{corollary}
\label{corollary: contractions tripod}
Let $\phi$ be tripod with weights as in \cref{corollary: weight formulas tripod}, and let 
\begin{equation}
a_i(s_0,s_1)
= 
\begin{vmatrix}
\phantom{-}s_0 & s_1 \\
-\alpha_{1i} & \alpha_{0i}
\end{vmatrix}
\ ,\ 
b_j(t_0,t_1)
= 
\begin{vmatrix}
\phantom{-}t_0 & t_1 \\
-\beta_{1j} & \beta_{0j}
\end{vmatrix}
\ ,\ 
c_k(u_0,u_1)
= 
\begin{vmatrix}
\phantom{-}u_0 & u_1 \\
-\gamma_{1k} & \gamma_{0k}
\end{vmatrix}
\end{equation}
for each $0\leq i,j,k\leq 1$, as well as 
$
h(s_0,s_1,t_0,t_1,u_0,u_1) 
= 
\omega_1 \, a_0 b_1 c_1 
+
\omega_2 \, a_1 b_0 c_1 
+
\omega_3 \, a_1 b_1 c_0 
$ 
for some $\omega_1,\omega_2,\omega_3$ in $\R$ such that 
$
q = 
\omega_1 \, \pi_2\pi_3
+
\omega_2 \, \pi_1\pi_3
+
\omega_3 \, \pi_1\pi_2 
$. 
Then, we have the following:
\vskip2pt
\begin{enumerate}
\item $\phi$ contracts the surface $a_1 = 0$ (resp$.$ $b_1 = 0$ and $c_1 = 0$) to the line $s$  (resp$.$ $t$ and $u$)
\item $\phi$ contracts the surface $h = 0$ to the plane conic $C$
\item The image of the surface $a_0 = 0$ (resp$.$ $b_0 = 0$ and $c_0 = 0$) is dense in the plane  $\Pi_1$ (resp$.$ $\Pi_2$ and $\Pi_3$)
\end{enumerate}
\end{corollary}

\begin{proof}
Follows from the proof of \cref{theorem: tensor tripod}. 
More explicitly for $2$, 
since $\phi$ contracts the surfaces $a_1 = 0$, $b_1 = 0$, $c_1 = 0$ respectively to $s,t,u$, we find the factorization 
$$
q(f_0,f_1,f_2,f_3) 
= 
a_1 b_1 c_1 
\, 
(
\omega_1 \, a_1 b_0 c_0
+
\omega_2 \, a_0 b_1 c_0
+
\omega_3 \, a_0 b_0 c_1
)
= 
a_1 b_1 c_1 
h
\ .
$$
By the proof of 3 implies 1 in \cref{theorem: tensor tripod}, the image of $h = 0$ lies on the plane $\Pi$. 
Therefore, it lies in $C = \Pi \cap Q$.  
\end{proof}

\subsection{Inverse and base locus}

We close the section giving explicit formulas for the inverse.

\begin{theorem} 
\label{theorem: inverse tripod}
Let $\phi$ be tripod with weights as in \cref{corollary: weight formulas tripod}. 
Then, $\phi^{-1}$ is given by 
$$
(x_0 : x_1 : x_2 : x_3) 
\mapsto 
\left(
\alpha_{1i} \, \lambda_{1i} \, \sigma_1
: 
\alpha_{0i} \, \lambda_{0i} \, \sigma_0
\right)
\times
\left( 
\beta_{1j} \, \mu_{1j} \, \tau_1
: 
\beta_{0j} \, \mu_{0j} \, \tau_0
\right)
\times 
\left(
\gamma_{1k} \, \nu_{1k} \, \upsilon_1
: 
\gamma_{0k} \, \nu_{0k} \, \upsilon_0
\right)
\ ,
$$
where any $0\leq i,j,k\leq 1$ are valid. 
With the notation of \cref{corollary: contractions tripod}, we have the following: 
\vskip2pt 
\begin{enumerate}
\item The base locus of $\phi$ is defined by the ideal  
\begin{equation}
\label{eq: base ideal tripod}
(a_0 , b_0 , c_0) \cap (a_1^2 , b_1^2 , c_1^2 , a_1 b_1 , a_1 c_1 , b_1 c_1 , \omega_1 \, a_1 b_0 c_0 + \omega_2 \, a_0 b_1 c_0 + \omega_3 \, a_0 b_0 c_1 )
\end{equation}
\item The base locus of $\phi^{-1}$ is $s\cup t\cup u\cup C$
\item $\phi$ blows up the base point $(a_0,b_0,c_0)$ to the plane $\Pi$
\item $\phi$ blows up the nonreduced base point $(a_1,b_1,c_1)$ to the cone $Q$
\end{enumerate}
\end{theorem}

\begin{proof}
By \cref{corollary: weight formulas tripod}, $\phi$ is birational of type $(2,2,2)$. 
As explained in the proofs of \cref{theorem: inverse pyramidal} and 
\cref{theorem: inverse scaffold}, the inverses on each factor of $\PPP$ can be regarded as line isomorphisms $\sigma,\tau,\upsilon:\P^1 \xrightarrow{} (\P^9)^\vee$. 
By \cref{corollary: contractions tripod}, 
$\phi$ contracts the surface $a_1 = 0$ (resp$.$ $b_1 = 0$ and $c_1 = 0$) to the line $s \subset Q$ (resp$.$ $t\subset Q$ and $u\subset Q$). 
Moreover, $\phi$ also maps the surface $a_0 = 0$ (resp$.$ $b_0 = 0$ and $c_0 = 0$) to the plane $\Pi_1$ (resp$.$ $\Pi_2$ and $\Pi_3$). 
These observations readily imply 
$
\sigma
(-\alpha_{10}:\alpha_{00}) 
=
\pi_1 \pi
$ and 
$
\sigma
(-\alpha_{11}:\alpha_{01}) 
=
q
$, 
since $\Pi_1\cup\Pi$ and $Q$ the unique $s$-surfaces containing respectively $\Pi_1$ and $s$. 
Hence, $\sigma:\P^1 \xrightarrow{} (\P^9)^\vee$ is explicitly given by  
$$
(s_0:s_1)
\mapsto 
\begin{vmatrix}
s_0 & s_1 \\ 
\alpha_{1i} \, \lambda_{1i} \, \sigma_1 & \alpha_{0i} \, \lambda_{0i} \, \sigma_0
\end{vmatrix}
$$
where both $i = 0,1$ are admissible. 
With parallel arguments, we derive explicitly $\tau$ and $\upsilon$. 

The base locus of $\phi^{-1}$ follows from \cref{type tripod}. 
Moreover, letting $I$ be \cref{eq: base ideal tripod} we find
$$
I_{(1,1,1)} 
= 
\R
\langle 
a_0 b_1 c_1 
,
a_1 b_0 c_1 
,
a_1 b_1 c_0 
, 
h
\rangle 
= 
\R
\langle 
\langle 
\bs{\pi}_1 , \mbf{f} 
\rangle 
, 
\langle 
\bs{\pi}_2 , \mbf{f} 
\rangle 
, 
\langle 
\bs{\pi}_3 , \mbf{f} 
\rangle 
, 
\langle 
\bs{\pi} , \mbf{f} 
\rangle 
\rangle 
= 
\R\langle f_0,f_1,f_2,f_3\rangle
\ .
$$
Hence, $I$ defines the base locus of $\phi$.  
The blow-ups follow from the proof of \cref{type tripod}. 
\end{proof}

\section{Effective manipulation of birational trilinear volumes}
\label{sec: manip}

We conclude the paper introducing several
effective methods for approximating and manipulating birational trilinear volumes.

\subsection{Distance to birationality}
\label{subsec: distbir}

The intuition of \cref{Q: distance} is clear, but the question itself lacks precision.
Specifically, we require a notion of ``distance to birationality'' to precise the term ``far''.
Interestingly, 
\cref{theorem: tensor hexahedral}, \cref{theorem: tensor pyramidal} - \cref{theorem: tensor tripod} provide a way to quantify this distance.

\begin{definition}
\label{definition: distance to bir}
Let $\phi$ be either hexahedral, pyramidal, scaffold, or tripod, and let $W$ be any of the tensors characterizing birationality. 
We define the \textit{distance to birationality} of $\phi$ as
\begin{equation}
\distbir(\phi) 
\coloneqq
\min_{P\in V}
\,
\lVert W \rVert^{-1}\lVert W - P \rVert
\ ,
\end{equation}
where $V$ is the affine cone over the Segre embedding of $\P_\R^1\times \P_\R^1\times \P_\R^1$.
\end{definition}


Within this formulation, a ``closest birational map'' refers to a point in $V$ that minimizes the distance to $W$. 
In order to compute a birational approximation of $\phi$, we can thus solve the problem 
$$
\min_{\substack{W_\text{bir}}} \, \lVert W \rVert^{-1}\lVert W - W_\text{bir} \rVert \text{ where } 
W_\text{bir} = \bs{\alpha} \otimes \bs{\beta} \otimes \bs{\gamma}
\text{ for some }
\bs{\alpha}, 
\bs{\beta},
\bs{\gamma}
\text{ in }
\R^2 
\ .
$$
There is an extensive body of work related to tensor rank and fixed-rank approximations to solve this problem. 
An important tool is the Canonical Polyadic Decomposition (CPD) (e$.$g$.$~\cite{CPD_2,tensorbook}). 
In \cref{example: birational approx} we illustrate how CPD can be used in practice to compute birational approximations. 

\subsection{Deformation of birational maps} 

An important task in the manipulation of birational maps is the continuous variation of the defining parameters while preserving birationality. 
The reason is that applications typically require to move the control points until the designer achieves the desired shape (e$.$g$.$ \cite{sederberg2D,SGW16,
construction_quadratic_Cremona_2,
construction_quadratic_Cremona_1}). 

Let $1\leq d_1,d_2,d_3\leq 2$ and $\mathcal{P}_{(d_1,d_2,d_3)}$ be the set of configurations of control points that are either hexahedral, pyramidal, scaffold, or tripod. 
From our results, it follows that the rational map 
\begin{eqnarray*}
	\Phi_{(d_1,d_2,d_3)}: (\P_\R^1)^3 \times \mathcal{P}_{(d_1,d_2,d_3)} & \dasharrow & \mathcal{W} \times \mathcal{P}_{(d_1,d_2,d_3)} 
	\\ 
\nonumber
(\alpha_0:\alpha_1) \times (\beta_0:\beta_1) \times (\gamma_0:\gamma_1) \times (\ldots \, ,\, \mbf{P}_{ijk} \, ,\,  \ldots) & \mapsto &
(\ldots : w_{ijk} : \ldots) \times (\ldots \, ,\, \mbf{P}_{ijk} \, ,\,  \ldots) 
= 
\\ \nonumber
	& &  
(\ldots : \Delta_{ijk} \, \alpha_i \, \beta_j \, \gamma_k : \ldots) \times (\ldots \, ,\, \mbf{P}_{ijk} \, ,\, \ldots)
\end{eqnarray*} 
is dominant on the component of trilinear birational maps of type $(d_1,d_2,d_3)$. 
In particular, this parametrization explains how deformation can be performed.
Specifically, a user can continuously decide new control points using \cref{construction: hexahedral} - \cref{construction: tripod}, and update the $\Delta_{ijk}$'s accordingly (which are continuous functions in the coordinates of the points) while keeping the same tensor decomposition  
$W = (\alpha_0, \alpha_1) \otimes (\beta_0, \beta_1) \otimes (\gamma_0, \gamma_1)$ for each instant of the deformation by applying the formulas $w_{ijk} = \alpha_i \, \beta_j \, \gamma_k \, \Delta_{ijk}$ for each $0\leq i,j,k\leq 1$. 

To illustrate, in \cref{example: deformation} we present a deformation of hexahedral birational maps. 
This case is particularly interesting, since hexahedral meshes appear frequently in geometric modeling.
Furthermore, the generation of the control points  in \cref{construction: hexahedral} is simple and intuitive.


\bibliographystyle{siamplain}
\bibliography{M166413}

\newpage

\appendix 

\section{Supplementary Materials}

In these supplementary materials, we discuss the methods introduced in the paper, illustrating them through several examples that emphasize their significance for various applications.

\subsection{Dicussion of the distance to birationality}

Let $\mathfrak{Rat}_{(1,1,1)}$ denote the space of trilinear rational maps, and let $\mathfrak{Bir}_{(1,1,1)}\subset \mathfrak{Rat}_{(1,1,1)}$ be the locus of birational maps. 

In full generality, measuring the distance from a trilinear rational map $\phi$ to $\mathfrak{Bir}_{(1,1,1)}$ presents a delicate challenge. 
The first complication arises from the high dimensionality, since $\phi$ is determined by $31$ independent parameters (in the monomial basis: $8\times 4$ coefficients up to scalar; in the Bernstein basis: $3$ coordinates $\times$ $8$ control points + $8$ weights up to scalar) and hence $\mathfrak{Rat}_{(1,1,1)}\cong \P^{31}$. 
Secondly, by \cite[Theorem 4$.$1]{trilinear} $\mathfrak{Bir}_{(1,1,1)}$ has eight irreducible components of various dimensions, and the distance from  $\phi$ to each of the components is in general different. 

Remarkably, the notion of distance introduced in \cref{definition: distance to bir} depends solely on the weights of $\phi$ (8 parameters up to scalar).
This is a design-oriented approach, since a designer will typically move the control points of a rational map but will rarely modify the weights.

\begin{remark}
The following observations are important:
\begin{enumerate}
\item For scaffold and tripod maps, several choices for $W$ are possible. However, \cref{theorem: tensor scaffold} and \cref{theorem: tensor tripod} assert that $\distbir(\phi) = 0$ represents the same condition for any choice. 
\item $\distbir(\phi)$ is well defined. More explicitly, altough the vectors of $w_{ijk}$'s and $\Delta_{ijk}$'s are defined up to a nonzero scalar, they yield  proportional tensors. Hence, the relative distances coincide. 
\end{enumerate}
\end{remark}

\subsection{Applications}
\label{examples}

The following example illustrates how a tripod rational map is generated.

\begin{example}
\label{example: quadcube}
Recall that we identify the vector 
$
\setlength\arraycolsep{1pt}
\begin{pmatrix}
1, & x, & y, & z
\end{pmatrix}
$ in $\R^4$ with the point $(x,y,z)$ in the affine chart $\A_\R^3\subset\P_\R^3$ defined by $x_0 = 1$. 
Let 
$
\setlength\arraycolsep{1pt}
\mbf{A} 
= 
\begin{pmatrix}
1, & 0, & 0, & 0
\end{pmatrix}$, 
and the lines in $\P_\R^3$ defined by 
$$
s : x_2 = x_3 = 0
\ , \ 
t : x_1 = x_3 = 0
\ ,\ 
u : x_1 = x_2 = 0
\ .
$$
Similarly, consider the plane conic $C$ 
in $\P_\R^3$ defined by the ideal 
$(
3 x_0 - x_1 - x_2 - x_3 
\ ,\ 
x_2 x_3 + x_1 x_3 + x_1 x_2
) 
$. 
If we choose the control point 
$
\setlength\arraycolsep{1pt}
\mbf{P}_{000} 
= 
\begin{pmatrix}
1, & \frac{1}{4}, & \frac{1}{4}, & \frac{1}{4}
\end{pmatrix}$, 
\cref{construction: tripod} yields the boundary lines 
$$
s_{00} = \ovl{\mbf{P}_{000}\mbf{P}_{100}}
\ ,\ 
t_{00} = \ovl{\mbf{P}_{000}\mbf{P}_{010}}
\ ,\ 
u_{00} = \ovl{\mbf{P}_{000}\mbf{P}_{001}}
\ ,
$$
where we have taken 
\begin{eqnarray*}
\setlength\arraycolsep{1pt}
\mbf{P}_{100} 
= 
\begin{pmatrix}
1, & \frac{3}{16}, & \frac{27}{80}, & \frac{27}{80}
\end{pmatrix}
\ ,\
\mbf{P}_{010} 
= 
\begin{pmatrix}
1, & \frac{27}{80}, & \frac{3}{16}, & \frac{27}{80}
\end{pmatrix}
\ ,\
\mbf{P}_{001} 
= 
\begin{pmatrix}
1, & \frac{27}{80}, & \frac{27}{80}, & \frac{3}{16}
\end{pmatrix}
\ .
\end{eqnarray*}
Continuing with \cref{construction: tripod}, 
we compute the remaining control points, 
\begin{align*}
\setlength\arraycolsep{1pt}
\mbf{P}_{110} 
= 
\begin{pmatrix}
1, & \frac{45}{172}, & \frac{45}{172}, & \frac{81}{172}
\end{pmatrix}
\ ,&\
\setlength\arraycolsep{1pt}
\mbf{P}_{101} 
= 
\begin{pmatrix}
1, & \frac{45}{172}, & \frac{81}{172}, & \frac{45}{172}
\end{pmatrix}
\ ,\\
\setlength\arraycolsep{1pt}
\mbf{P}_{011} 
= 
\begin{pmatrix}
1, & \frac{81}{172}, & \frac{45}{172}, & \frac{45}{172}
\end{pmatrix}
\ ,&\
\setlength\arraycolsep{1pt} 
\mbf{P}_{111} 
= 
\begin{pmatrix}
1, & \frac{3}{8}, & \frac{3}{8}, & \frac{3}{8}
\end{pmatrix}
\ .
\end{align*}

In \cref{quadcube}, we showcase the deformation of the unit cube $[0,1]^3$ induced by $\phi$, with uniform weights $w_{ijk} = 1$ for every $0\leq i,j,k \leq 1$. 
Additionally, we depict the lines involved in the generation of the control points.
\begin{figure}
\begin{minipage}{.32\textwidth}
	\begin{center}	\includegraphics[width=0.95\textwidth]{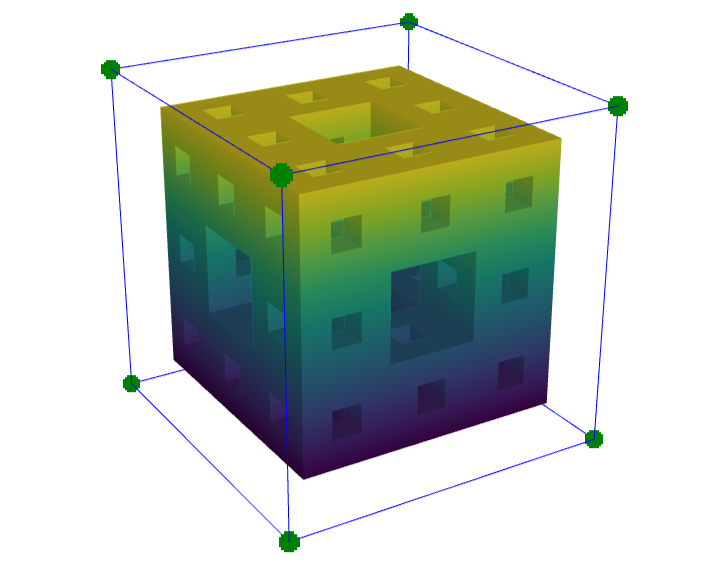}
	\end{center}
\end{minipage}
\begin{minipage}{.32\textwidth}
	\begin{center}
	\includegraphics[width=.93\textwidth]{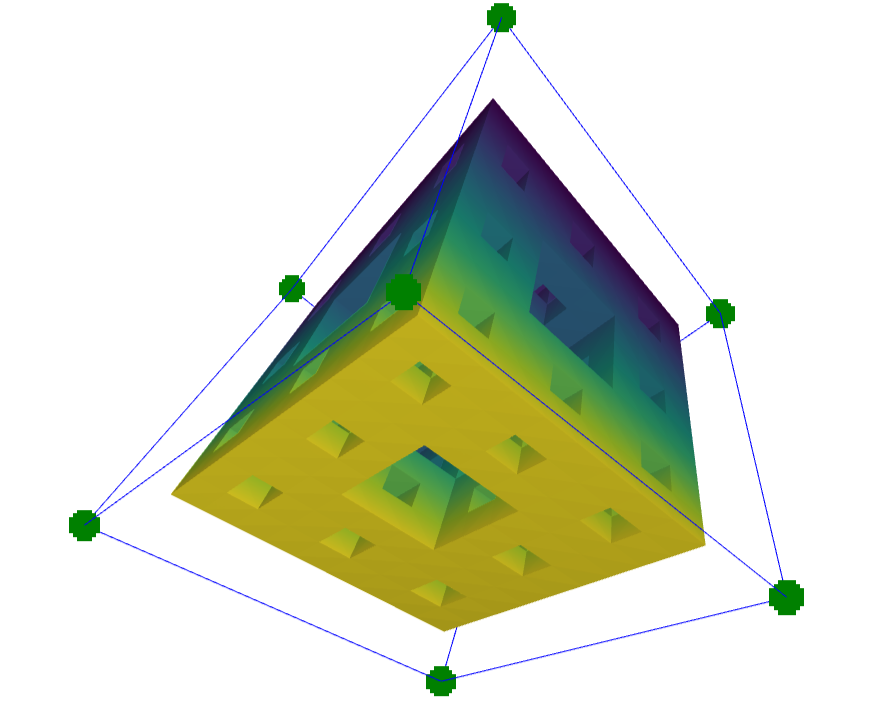}
	\end{center}
\end{minipage}
\begin{minipage}{.32\textwidth}
	\begin{center}
	\includegraphics[width=.93\textwidth]{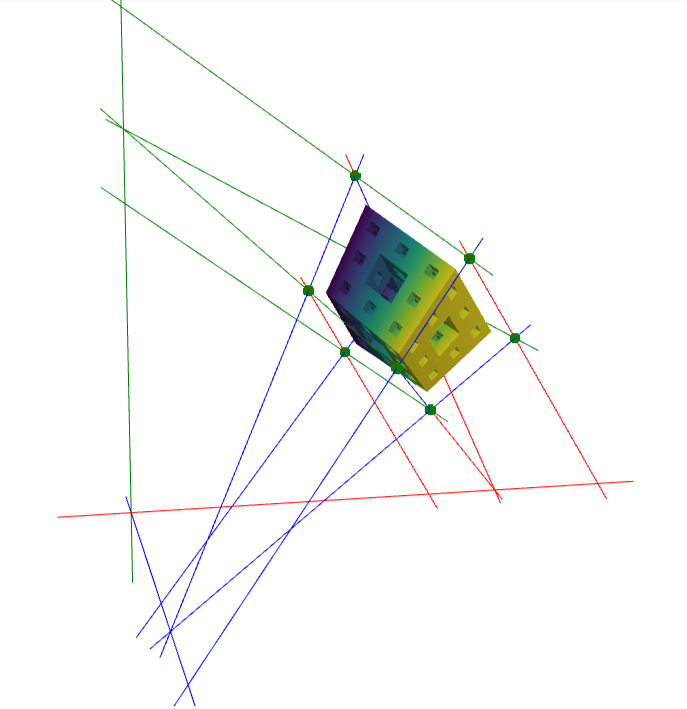}
	\end{center}
\end{minipage}
\caption{Left: the unit cube enclosing a model. Middle: deformation of the unit cube and the model using the trilinear rational map constructed in \cref{example: quadcube}, with uniform weights $w_{ijk} = 1$ for every $0\leq i,j,k \leq 1$. 
Right: the lines $s$ (red), $t$ (green), and $u$ (blue) used in the construction. These three lines intersect at a common point $\mbf{A}$. Additionally, all the boundary $s$-lines (red), $t$-lines (green), and $u$-lines respectively intersect $s,t$, and $u$.}
\label{quadcube}
\end{figure}
In particular, the boundary surfaces are smooth quadrics. 
Hence, these control points produce a ``quadcube'', i$.$e$.$ a quadrilateral-faced hexahedron where the flat facets are replaced by quadric surfaces.
\end{example}

In the setting of \cref{subsec: distbir}, a closest birational map to $\phi$ is an orthogonal projection of $\phi$ onto $V$.  
For a general $\phi$, there are exactly $6$ such projections, where at least two of them are real. 
Equivalently, the Euclidean distance degree of $V$ is $6$ \cite[Example 8$.$2]{EDdegree}.
In practice, a CP decomposition of the tensor governing birationality can be employed to compute a birational approximation, as demonstrated in the following example.

\begin{example}[Birational approximation of a pyramidal rational map]
\label{example: birational approx}
Consider the trilinear rational map $\phi$ defined by the control points
\begin{align*}
\setlength\arraycolsep{1pt}
\mbf{P}_{000} = 
\begin{pmatrix}
1, & -\frac{8}{5}, & 0, & 1
\end{pmatrix}
\ ,\ 
\mbf{P}_{100} = 
\begin{pmatrix}
1, & 0, & -\frac{9}{5}, & \frac{1}{2}
\end{pmatrix}
\ ,\ 
\mbf{P}_{010} = 
\begin{pmatrix}
1, & 0, & \frac{27}{20}, & \frac{1}{2}
\end{pmatrix}
\ ,\ 
\mbf{P}_{110} = 
\begin{pmatrix}
1, & \frac{4}{5}, & 0, & 1
\end{pmatrix}
\ ,\\[6pt] 
\setlength\arraycolsep{1pt}
\mbf{P}_{001} = 
\begin{pmatrix}
1, & -\frac{11}{10}, & 0, & \frac{9}{4}
\end{pmatrix}
\ ,\ 
\mbf{P}_{101} = 
\begin{pmatrix}
1, & 0, & -\frac{4}{5}, & 3
\end{pmatrix}
\ ,\ 
\mbf{P}_{011} = 
\begin{pmatrix}
1, & 0, & \frac{3}{5}, & 3
\end{pmatrix}
\ ,\ 
\mbf{P}_{111} = 
\begin{pmatrix}
1, & \frac{11}{20}, & 0, & \frac{9}{4}
\end{pmatrix}
\ ,
\end{align*}
and $w_{ijk} = 1$ for every $0\leq i,j,k \leq 1$. 
Since the four boundary $u$-lines, i$.$e$.$ $u_{ij} = \ovl{\mbf{P}_{ij0}\mbf{P}_{ij1}}$ for each $0\leq i,j \leq 1$, meet at the point 
$\mbf{A} = 
\setlength\arraycolsep{1pt}
\begin{pmatrix}
1, & 0, & 0, & 5
\end{pmatrix}
$
it follows that $\phi$ is pyramidal. 
In particular, we compute (recall \cref{notation: pyramidal})
\begin{gather*}
\Delta_{000} = \frac{5}{6} \ ,\ 
\Delta_{100} = \frac{20}{21} \ ,\  
\Delta_{010} = \frac{80}{63} \ ,\  
\Delta_{110} = \frac{5}{3} \ ,\\[6pt]   
\Delta_{001} = \frac{40}{33} \ ,\ 
\Delta_{101} = \frac{15}{7} \ ,\ 
\Delta_{011} = \frac{20}{7} \ ,\  
\Delta_{111} = \frac{80}{33} \ . 
\end{gather*}
Thus, the tensor 
$$
W 
= \left( \frac{w_{ijk}}{\Delta_{ijk}}\right)_{0\leq i,j,k\leq 1}
= \left( \frac{1}{\Delta_{ijk}}\right)_{0\leq i,j,k\leq 1}
$$
does not have rank one, and by  \cref{theorem: tensor pyramidal} $\phi$ is not birational. 
However, we can render $\phi$ birational by computing new weights using the formulas in \cref{corollary: weight formulas pyramidal} for any choice of $(\alpha_0:\alpha_1)\times (\beta_0:\beta_1)\times (\gamma_0:\gamma_1) \in (\P_\R^1)^3$. 
Nevertheless, the new birational map $\phi$ might differ significantly from the original $\phi$ if we make a ``bad'' choice. 
A smarter idea is to compute a rank-one CP decomposition of $W$. 
A rank-one approximation of $W$ is given by 
$$
W_{\text{bir}} = 
\left(
\alpha_0 \, ,\, \alpha_1
\right)
\otimes
\left(
\beta_0 \, ,\, \beta_1
\right)
\otimes
\left(
\gamma_0 \, ,\, \gamma_1
\right)
=
\left( 
0.95 \, ,\, 0.91
\right)
\otimes
\left( 
1.06 \, ,\, 0.78
\right)
\otimes
\left( 
1.08 \, ,\, 0.75
\right)
\ .
$$
Specifically, this leads to the exact rational weights
$
w_{ijk} = \alpha_i \, \beta_j \, \gamma_k \, \Delta_{ijk}
$
for each $0 \leq i,j,k \leq 1$. 
Additionally, we compute  
$$
\distbir(\phi) = 
\frac{\lVert W - W_\text{bir} \rVert}{\lVert W \rVert} 
\sim 
0.2594
\ .
$$
In \cref{fig: birational approx}, we present a comparison between the deformations resulting from the original rational map, with uniform weights $w_{ijk} = 1$ for every $0\leq i,j,k \leq 1$, and its  birational approximation, utilizing the new weights. 
\begin{figure}
\begin{minipage}{.48\textwidth}
	\begin{center}
	\includegraphics[width=0.8\textwidth]{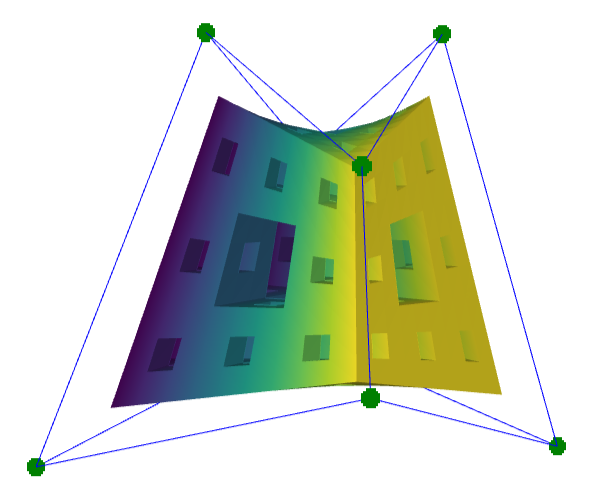}
	\end{center}
\end{minipage}
\begin{minipage}{.48\textwidth}
	\begin{center}
	\includegraphics[width=0.8\textwidth]{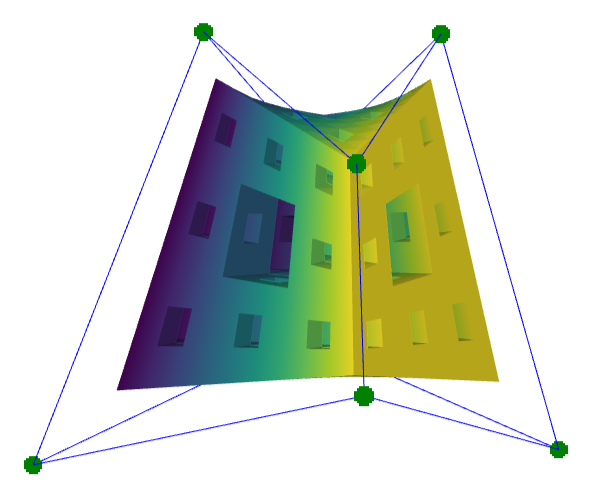}
	\end{center}
\end{minipage}
\caption{The two pyramidal maps presented in \cref{example: birational approx}. 
In the left image, we use uniform weights, specifically $w_{ijk} = 1$ for every $0\leq i,j,k \leq 1$. 
The right image showcases the effect of utilizing the computed birational weights. 
These weight adjustments subtly influence the deformation, while simultaneously ensuring the existence of an inverse transformation.}
\label{fig: birational approx}
\end{figure}
\end{example}
   
Moreover, our results directly yield the defining formulas of the inverse rational map. 
 
\begin{example}
\label{chapter: introduction - section: contributions - example: pyramidal 2}
We continue with \cref{example: birational approx}, using the computed birational weights.  
In this case, the vectors and matrices defining the boundary surfaces are 
\begin{align*}
\setlength\arraycolsep{2pt}
\bs{\sigma}_0 = 
\begin{pmatrix}
5, & -\frac{5}{2}, & \frac{10}{3}, & 1
\end{pmatrix} \ \ ,\ \ & 
\setlength\arraycolsep{2pt}
\bs{\sigma}_1 = 
\begin{pmatrix}
4, & -4, & 2, & -\frac{4}{5}
\end{pmatrix}
\ ,\\[6pt]
\setlength\arraycolsep{2pt}
\bs{\tau}_0 = 
\begin{pmatrix}
5, & \frac{5}{2}, & \frac{5}{2}, & -1
\end{pmatrix} \ \ ,\ \ & 
\setlength\arraycolsep{2pt}
\bs{\tau}_1 = 
\begin{pmatrix}
-3, & 3, & 2, & \frac{3}{5}
\end{pmatrix}
\ ,\\[6pt]
\bs{\upsilon}_0 = 
\begin{pmatrix}
390 & -30 & -\frac{45}{2} & -294 \\[4pt] 
-30 & 75 & 0 & 60 \\[4pt] 
-\frac{45}{2} & 0 & -50 &\frac{45}{2} \\[4pt]
-294 & 60 & \frac{45}{2} & 102
\end{pmatrix} \ \ ,\ \ & 
\bs{\upsilon}_1 = 
\begin{pmatrix}
1530 & -\frac{495}{2} & -45 & -438 \\[4pt] 
-\frac{495}{2} & -225 & 0 & \frac{165}{2} \\[4pt] 
-45 & 0 & 150 & 20 \\[4pt]
-438 & \frac{165}{2} & 20 & 114
\end{pmatrix}
\ .
\end{align*} 
Furthermore, using the formulas in  \cref{notation: pyramidal} we compute 
\begin{align*}
(\lambda_0:\lambda_1)
\times 
(\mu_0:\mu_1)
\times 
(\nu_0:\nu_1)
= 
(4 : -5)
\times 
(3 : -5)
\times 
(11 : -18)
\ . 
\end{align*}
Thus, by \cref{theorem: inverse pyramidal} $\phi^{-1}:\P_\R^3\dashrightarrow \P_\R^1 \times \P_\R^1\times \P_\R^1$ is given by  
$$
(x_0 : x_1 : x_2 : x_3) 
\mapsto 
\left(
\alpha_1 \, \lambda_1 \, \sigma_1
: 
\alpha_0 \, \lambda_0 \, \sigma_0
\right)
\times
\left( 
\beta_1 \, \mu_1 \, \tau_1
: 
\beta_0 \, \mu_0 \, \tau_0
\right)
\times 
\left(
\gamma_1 \, \nu_1 \, \upsilon_1
: 
\gamma_0 \, \nu_0 \, \upsilon_0
\right)
\ .
$$
\end{example} 

We conclude with an example of the deformation of hexahedral birational maps.

\begin{figure}
\begin{minipage}{.48\textwidth}
	\begin{center}
	\includegraphics[scale=0.35]{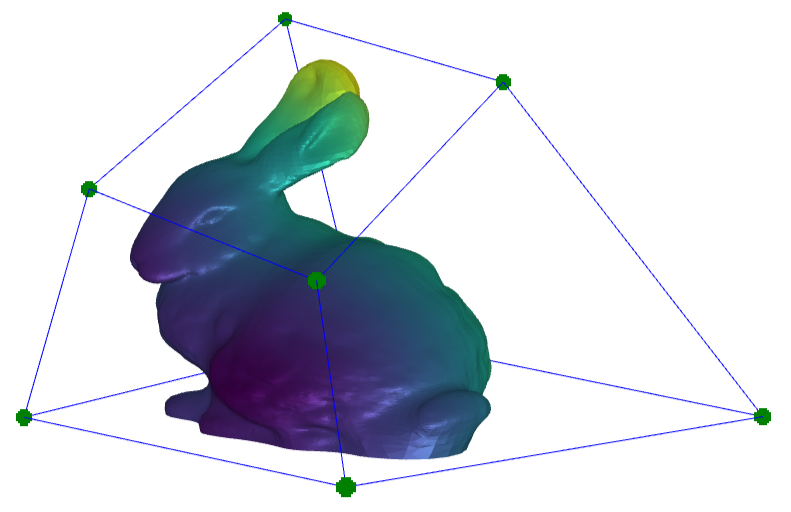}
	\end{center}
	\end{minipage}
	\begin{minipage}{.48\textwidth}
	\begin{center}
	\includegraphics[scale=0.42]{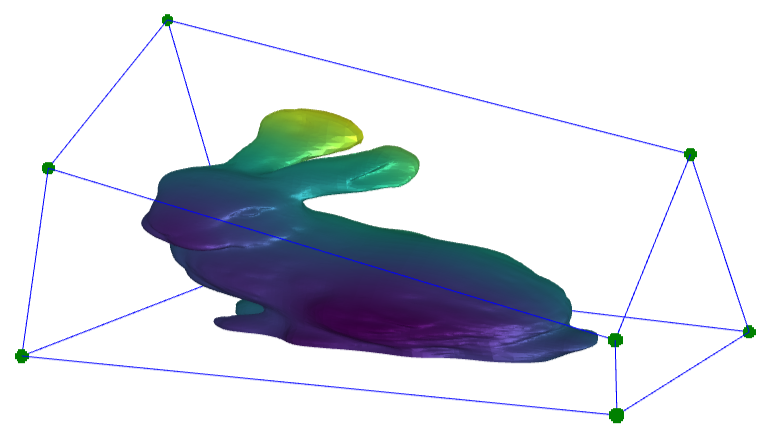}
	\end{center}
	\end{minipage}
\caption{A birational deformation of the hexahedral birational map of \cref{example: deformation}. 
The boundary plane $\Sigma_1$, previously defined by the vector 
$\bs{\sigma}_1 = 
(
1.25, \, -0.63, \, -0.32, \, -0.63 
)$, is updated to the plane defined by 
$\bs{\sigma}_1 = 
(
2.31, \, -0.84, \, -0.2, \, -0.32 
)$, yielding new control points.
Additionally, the weights are also updated to preserve birationality.}
\label{figure: deformation} 
\end{figure}

\begin{example}[Deformation of a hexahedral birational map]
\label{example: deformation}
Consider the boundary planes defined by
\begin{align*}
\setlength\arraycolsep{2pt}
\bs{\sigma}_0 = 
\begin{pmatrix}
0.16, & -0.45, & -0.07, & -0.14
\end{pmatrix} \ \ ,\ \ & 
\setlength\arraycolsep{2pt}
\bs{\sigma}_1 = 
\begin{pmatrix}
1.25, & -0.63, & -0.32, & -0.63 
\end{pmatrix}
\ ,\\[6pt]
\setlength\arraycolsep{2pt}
\bs{\tau}_0 = 
\begin{pmatrix}
0, & 0, & 0, & 1
\end{pmatrix} \ \ ,\ \ & 
\setlength\arraycolsep{2pt}
\bs{\tau}_1 = 
\begin{pmatrix}
 -1.18, & 0.18, & 0.51, 1
\end{pmatrix}
\ ,\\[6pt]
\setlength\arraycolsep{2pt}
\bs{\upsilon}_0 = 
\begin{pmatrix}
0, & 0, & 1, & 0 
\end{pmatrix} \ \ ,\ \ & 
\setlength\arraycolsep{2pt}
\bs{\upsilon}_1 = 
\begin{pmatrix}
-1.17, & 0.1, & 0.8, & 0.54  
\end{pmatrix}
\ .
\end{align*} 
For each $0\leq i,j,k \leq 1$, we can express $\mbf{P}_{ijk} = \Sigma_i \cap T_j \cap Y_k$ as the exterior product 
\begin{equation}
\label{chapter: introduction - section: contributions - eq: control points as exterior product}
\mbf{P}_{ijk} = \Delta_{ijk}^{-1} \ \bs{\sigma}_i \wedge \bs{\tau}_j \wedge \bs{\upsilon}_k
\ , 
\end{equation}
where each $\Delta_{ijk}$ can be computed using \cref{eq: Deltas hexahedral}. 
If we initialize the weights as $w_{ijk} = 1$ for every $0\leq i,j,k \leq 1$, a rank-one approximation of the tensor $W$ in \cref{theorem: tensor hexahedral} can be computed with the same strategy of \cref{example: birational approx}, 
\begin{gather*}
W = 
(\alpha_0 \, ,\, \alpha_1)
\otimes 
(\beta_0 \, ,\, \beta_1)
\otimes 
(\gamma_0 \, ,\, \gamma_1)
= 
(1.56 \, ,\, 1.24)
\otimes 
(1.12 \, ,\, 1.65)
\otimes 
(1.02 \, ,\, 1.71)
\ .
\end{gather*}
By \cref{corollary: weight formulas hexahedral}, setting $w_{ijk} = \alpha_i \, \beta_j \, \gamma_k \, \Delta_{ijk}$ 
renders $\phi$ birational. 
The left image in \cref{figure: deformation} shows the deformation of a model enclosed in the unit cube $[0,1]^3$ by $\phi$. 

Now, suppose that we want to update the control points, as a user frequently does during the design process.
However, we want to do so preserving birationality, so the property of being hexahedral must be preserved (recall \cref{fact: hexahedral}). 
Using \cref{construction: hexahedral}, we decide to update the boundary plane $\bs{\sigma}_1$ linearly in time, namely 
\begin{align*}
t \mapsto\ & 
\bs{\sigma}_1(t) = 
\setlength\arraycolsep{2pt}
(1-t)
\begin{pmatrix}
1.25 & -0.63 & -0.32 & -0.63 
\end{pmatrix} 
+ 
t 
\begin{pmatrix}
2.31 & -0.84 & -0.2 & -0.32 
\end{pmatrix}
\end{align*} 
for $t\in [0,1]$.  
In particular, for each $0\leq j,k\leq 1$ this induces a linear  function $\Delta_{1jk}(t)$. 
Therefore, setting 
$$
w_{1jk}(t) = \alpha_1 \, \beta_j \, \gamma_k \, \Delta_{1jk}(t)
$$
ensures birationality for each instant $t\in[0,1]$ of the deformation. 
\end{example}

\end{document}